%% file: HHtannaka2.tex
\documentclass[11pt,twoside]{amsart}

\usepackage{latexsym}
\usepackage{amssymb}
\usepackage{vmargin}
\usepackage{amscd}
\usepackage{stmaryrd}
\usepackage{euscript}
\usepackage{mathrsfs}
\usepackage[all]{xy}
\usepackage{xr}
\usepackage{bbm}

\setmargins{32mm}{20mm}{14.6cm}{22cm}{1cm}{1cm}{1cm}{1cm}

\include{newcommands}

\include{newthms}

\begin{document}

\begin{abstract}
 We consider the effect of $t$-structures on the Tannaka duality theory for dg categories developed in \cite{HHtannaka1}. We  associate  non-negative dg coalgebras $C$ to dg functors on the hearts of $t$-structures, and relate dg $C$-comodules to the original dg category.  We give several applications for pro-algebraic homotopy types associated to various cohomology theories, and for motivic Galois groups.
\end{abstract}

\title[Tannaka duality for enhanced triangulated categories II: $t$-structures]{Tannaka duality for enhanced triangulated categories II: $t$-structures and homotopy types}
\author{J.P.Pridham}
\thanks{This work was supported by  the Engineering and Physical Sciences Research Council [grant numbers EP/I004130/1 and EP/I004130/2].}

\maketitle
\section*{Introduction}

Tannaka duality in Joyal and Street's formulation (\cite[\S 7, Theorem 3]{JoyalStreet}) characterises abelian $k$-linear categories $\cA$ with exact faithful $k$-linear functors $\omega$ to finite-dimensional $k$-vector spaces as categories of finite-dimensional comodules of coalgebras $C$. When $\cA$ is a rigid tensor category and $\omega$ monoidal, $C$ becomes a Hopf algebra (so $\Spec C$ is a group scheme), giving the duality theorem of \cite[Ch. II]{tannaka}.

In \cite{HHtannaka1}, these duality theorems are extended to dg categories. Given a $k$-linear dg category $\cA$ and a $k$-linear dg functor $\omega$ to finite-dimensional complexes, there is natural a dg coalgebra structure $C$ on the Hochschild homology complex
\[
 \omega^{\vee}\ten_{\cA}^{\oL}\omega\simeq \CCC_{\bt}(\cA, \omega^{\vee}\ten_k \omega).
\]
Then \cite[Theorem \ref{HHtannaka1-tannakathm}]{HHtannaka1}  shows that the    dg derived category $\cD_{\dg}(C)$ of $C$-comodules is quasi-equivalent to a derived quotient  $\cD_{\dg}(\cA)/(\ker\omega)$ of the
dg derived category $\cD_{\dg}(\cA)$   generated by $\cA$. In particular, when $\omega$ is faithful, this gives a quasi-equivalence
$\cD_{\dg}(\cA) \simeq \cD_{\dg}(C)$, which is a derived analogue of Joyal and Street's Tannaka duality.

The main drawback of the Hochschild construction for the dg coalgebra in \cite{HHtannaka1} is that it always creates terms in negative cochain degrees. This means that quasi-isomorphisms of such dg coalgebras might not be derived Morita equivalences, and that we cannot rule out negative homotopy groups for dg categories of cohomological origin. 

In Section  \ref{SSsn}, we give an alternative presentation of  the Hochschild construction which associates non-negative dg coalgebras to hearts of $t$-structures (Corollary \ref{coalgconc3}, Propositions \ref{coalgconcequiv},
 \ref{coalgconcper}).
In this setting, the correspondence between dg categories and dg coalgebras can be understood as a form of Koszul duality
(Proposition \ref{Koszulequiv}). Via duality of the commutative and Lie operads,  dg tensor categories then correspond to dg Hopf algebras (Corollary \ref{hopfalgconc3}, Proposition \ref{Koszulequiv2}). Example \ref{mothom2}  then explains how these results combine with Ayoub's calculations to show that existence of a motivic $t$-structure would characterise Voevodsky's motives over a number field as the derived category of Nori's abelian category of mixed motives, implying the $K(\pi,1)$ conjecture.
\S \ref{moriyasn} explains how our constructions generalise Moriya's Tannakian dg categories. 

Section \ref{htpysn} is mostly concerned with applications to the real relative Malcev homotopy types of a manifold $X$. 
Lemma \ref{integratelemma} equates the dg category of derived connections on $X$ with  the pre-triangulated category generated by the de Rham dg category of semisimple local systems. Corollary \ref{integratecor} then  equates this with the dg category of representations of the schematic homotopy type $G(X,x)^{\alg}$. \S \ref{univhalg} looks at the universal bialgebra, which avoids  choices of basepoint and can be thought of as the sheaf 
of functions on the space of algebraic paths.
In \S \ref{qlhtpy}, we establish analogues for $\Ql$ relative Malcev homotopy types of a scheme, and \S \ref{motsn} discusses motivic generalisations.

I would like to thank Joseph Ayoub for providing helpful comments and spotting careless errors.

\subsection*{Notational conventions}

 Fix a commutative ring $k$.
 When the base is not specified, $\ten$ will mean $\ten_k$. When $k$ is a field, we write $\Vect_k$ for the category of all vector spaces over $k$, and $\FD\Vect_k$ for the full subcategory of finite-dimensional vector spaces.

We will always use the symbol $\cong$ to denote isomorphism, while $\simeq$ will be equivalence,  quasi-isomorphism or quasi-equivalence.

\tableofcontents

\section{Background from \cite{HHtannaka1}}
We now recall some conventions, definitions and results from \cite{HHtannaka1}.

\subsection{Conventions for DG categories}
\begin{definition}
 A $k$-linear dg category $\C$ is a category enriched in cochain complexes of $k$-modules, so has objects $\Ob \C$, cochain complexes $\HHom_{\C}(x,y)$ of morphisms, associative multiplication
\[
 \HHom_{\C}(y,z)\ten_k\HHom_{\C}(x,y)\to \HHom_{\C}(x,z)
\]
and identities $\id_x \in \HHom_{\C}(x,x)^0$.
\end{definition}

Given a dg category $\C$, we will write $\z^0\C$ and $\H^0\C$ for the categories with the same objects as $\C$ and with morphisms
\begin{align*}
 \Hom_{\z^0\C}(x,y) &:=\z^0\HHom_{\C}(x,y),\\
\Hom_{\H^0\C}(x,y) &:=\H^0\HHom_{\C}(x,y).
\end{align*}

 When we refer to limits or colimits in a dg category $\C$, we will mean limits or colimits in the underlying category $\z^0\C$.

\begin{definition}
Given a dg category $\C$ and objects $x,y$, write $\C(x,y):= \HHom_{\C}(y,x)$.
\end{definition}

\begin{definition}
 A dg functor $F\co \cA \to \cB$ is said to be a quasi-equivalence if $\H^0F\co \H^0\cA \to \H^0\cB$ is an equivalence of categories, with $\cA(X,Y) \to \cB(FX,FY)$ a quasi-isomorphism for all objects $X,Y \in \cA$.
\end{definition}

\begin{definition}\label{perkdef}
We follow \cite[2.2]{kellerModelDGCat} in writing $\C_{\dg}(k)$ for the dg category of chain complexes over $k$, where $\HHom(U,V)^i$ consists of graded $k$-linear morphisms $U \to V[i]$, and the differential is given by $df= d \circ f \mp f \circ d$. 

We write
$\per_{\dg}(k)$ for the full dg subcategory of 
finite rank  cochain complexes of projective $k$-modules. Beware that this category is not closed under quasi-isomorphisms, so does not include all perfect complexes in the usual sense. 
\end{definition}

Following the conventions of \cite[3.1]{kellerModelDGCat}, we will write $\C_{\dg}(\cA)$ for the  dg category  of $k$-linear dg functors $\cA^{\op} \to \C_{\dg}(k)$ to  chain complexes over $k$. Observe that when $\cA$ has a single object $*$ with $\cA(*,*)=A$, $\C_{\dg}(\cA)$ is equivalent to the category of $A$-modules in complexes. We write $\C(\cA)$ for the (non-dg) category $\z^0\C_{\dg}(\cA)$ of dg $\cA$-modules.

An object $P$ of  $\C(\cA)$ is cofibrant  (for the projective model structure) if every surjective quasi-isomorphism $L \to P$ has a section. The full dg subcategory of $\C_{\dg}(\cA)$ on cofibrant objects  is  denoted $\cD_{\dg}(\cA)$. 
This is the idempotent-complete  pre-triangulated category (in the sense of \cite[Definition 3.1]{BondalKapranov}) generated by $\cA$ and closed under filtered colimits. We write  $\cD(\cA)$ for the derived category $\H^0\cD_{\dg}(\cA)$ of dg $\cA$-modules --- this is equivalent to the localisation of $\C(\cA)$ at quasi-isomorphisms. Thus $\cD_{\dg}(\cA)$ is a dg enhancement of the triangulated category $\cD(A)$.

\begin{definition}\label{perdef}
 Define $\per_{\dg}(\cA) \subset \cD_{\dg}(\cA)$ to be the full subcategory on compact objects, i.e those $X$ for which
\[
 \HHom_{\cA}(X,-)
\]
 preserves filtered colimits.
Explicitly, $\per_{\dg}(\cA)$  consists of objects arising as direct summands of finite complexes of objects of the form $h_{X}[n]$, for $X \in \cA$, where $h$ is the Yoneda embedding. 
\end{definition}

When $\cA$ has a single object $*$ with $\cA(*,*)=A$, then $h_{*}[n]$ corresponds to the $A$-module $A[n]$. Since projective modules are direct summands of free modules, Definitions \ref{perdef} and \ref{perkdef} are thus consistent.

As explained in \cite[4.5]{kellerModelDGCat},  $\per_{\dg}(\cA)$ is the idempotent-complete   pre-triangulated envelope or hull of $\cA$, in the sense of \cite[\S 3]{BondalKapranov}. 
Note that in \cite[\S 2]{kellerHCexact}, pre-triangulated categories are called exact DG categories. 

By \cite[Theorem 5.1]{tabuadaInvariantsAdditifs}, there is a Morita model structure on $k$-linear dg categories. Weak equivalences are dg functors $\cA \to \cB$ which are derived Morita equivalences in the sense that
\[
 \cD_{\dg}(\cA) \to \cD_{\dg}(\cB)       
\]
is a quasi-equivalence.
The dg functor $\cA \to \per_{\dg}(\cA)$ is  fibrant replacement in this model structure.

Note that a dg category $\cA$ is an  idempotent-complete pre-triangulated category if and only if the natural embedding $\cA \to \per_{\dg}(\cA)$ is a quasi-equivalence. This is equivalent to saying that $\cA$ is Morita fibrant (i.e. fibrant in the Morita model structure), or triangulated in the terminology of \cite[Definition 2.4]{TVdg}.

\subsection{Hochschild homology of a DG category}\label{HHsn}


The following is adapted from \cite[\S 12]{mitchellRings} and \cite[1.3]{kellerHCexact}:

\begin{definition}
Take a small $k$-linear dg category $\cA$ and an $\cA$-bimodule
\[
 F\co  \cA\by \cA^{\op}\to \C_{\dg}(k),       
\]
(i.e. a $k$-bilinear dg functor). Define the homological Hochschild complex
\[
 \uline{\CCC}_{\bt}(\cA, F)      
\]
(a simplicial diagram of cochain complexes) by
\[
 \uline{\CCC}_n(\cA,F):= \bigoplus_{X_0, \ldots, X_n \in \Ob \cA }  \cA(X_0,X_1)\ten_k\cA(X_1,X_2) \ten_k \ldots \ten_k\cA(X_{n-1},X_n) \ten_kF(X_n,X_0),   
\]
with face maps
\[
 \pd_i(a_1\ten\ldots a_n \ten f)= \left\{ \begin{matrix}   a_2\ten\ldots a_n \ten (f \circ a_1) & i=0 \\  a_1\ten\ldots a_{i-1} \ten (a_i \circ a_{i+1}) \ten a_{i+2} \ten \ldots\ten a_n \ten f & 0<i <n \\
  a_1\ten\ldots a_{n-1} \ten (a_n \circ f) & i=n                                              
                                          \end{matrix}\right.
 \]
and degeneracies
\[
 \sigma_i(a_1\ten\ldots a_n \ten f)= (a_1\ten \ldots \ten a_i \ten \id \ten a_{i+1} \ten \ldots \ten a_n \ten f).
 \]
\end{definition}

\begin{definition}


Define the total  Hochschild complex
\[
 \CCC(\cA,F)
\]
by first regarding $\uline{\CCC}_{\bt}(\cA, F)$ as a chain cochain complex with chain differential $\sum_i (-1)^i\pd_i$, then taking the total complex
\[
( \Tot \uline{\CCC}_{\bt}(\cA, F)^{\bt})^n = \bigoplus_i \uline{\CCC}_{i}(\cA, F)^{n+i},
\]
 with differential given by the cochain differential $\pm$ the chain differential. 

There is also a quasi-isomorphic normalised version 
\[
 N\CCC(\cA,F),
\]
given by replacing $\uline{\CCC}_{i}$ with $\uline{\CCC}_{i}/\sum_j \sigma_j \uline{\CCC}_{i-1}$.
\end{definition}

\begin{remark}
Note that $\H^i \CCC(\cA,F)^{\bt}= \HH_{-i}(\cA,F)$,  which is a Hochschild \emph{homology} group. We have, however, chosen cohomological gradings because our motivating  examples will all have $\H^{<0}=0$.
\end{remark}

 \subsubsection{The Tannakian envelope}

 Fix a small $k$-linear dg category $\cA$ and a $k$-linear dg  functor $\omega \co \cA \to \per_{\dg}(k)$.
 
%
\begin{definition}\label{Comegadef}
 Define the Tannakian dual $C_{\omega}( \cA)$ by 
\[
 C_{\omega}( \cA):= \CCC(\cA,\omega \ten \omega^{\vee}),
\]
where the $\cA$-bimodule
\[
 \omega \ten \omega^{\vee} \co \cA \by \cA^{\op} \to \per_{\dg}(k) 
\]
is given 
by
\[
 \omega \ten \omega^{\vee} (x,y)= (\omega x)\ten_k (\omega y)^{\vee}.
\]
Similarly, write $NC_{\omega}( \cA):= N\CCC(\cA,\omega \ten \omega^{\vee})$.
\end{definition}

By \cite[Proposition \ref{HHtannaka1-coalg}]{HHtannaka1},  The cochain complexes $C_{\omega}( \cA), NC_{\omega}( \cA) $ 
 have the natural structure of  coassociative counital dg coalgebras over $k$.

 \subsubsection{The universal coalgebra and tilting modules}\label{univcoalg}\label{tilt}
%
%
%
Take a $k$-linear dg category $\cA$, and  $D \in \cD_{\dg}(\cA^{\op}\ten \cA)$  a coassociative $\ten_{\cA}$-coalgebra, with the co-unit $D \to \id_{\cA}$
a  quasi-isomorphism. We regard this as being a universal coalgebra associated to $\cA$.

 \begin{example}\label{HHcoalgex}
As in \cite[Example \ref{HHtannaka1-HHcoalgex}]{HHtannaka1}, if the $k$-complexes $\cA(X, Y)$ are all cofibrant (automatic when $k$ is a field), canonical choices for $D$ are the unnormalised and normalised Hochschild complexes
\begin{align*}
   \CCC(\cA,h_{\cA^{\op}} \ten h_{\cA})\\ 
 N\CCC(\cA,h_{\cA^{\op}} \ten h_{\cA})
\end{align*}
    of the Yoneda embedding $h_{\cA^{\op}} \ten h_{\cA}  \co \cA^{\op}\ten \cA\to \C_{\dg}(\cA^{\op}\ten \cA)$. D
\end{example}

Given 
$\omega \co \cA \to \per_{\dg}(k)$, define the tilting module  $P$ by $P:=D\ten_{\cA}\omega \in \C(\cA^{\op})$; this is  cofibrant and has  a natural quasi-isomorphism $P \to \omega$. Also set $Q\in \C(\cA)$ by $Q:=\omega^{\vee}\ten_{\cA}D$ and set $C:= \omega^{\vee}\ten_{\cA}D\ten_{\cA}\omega$. Note that the natural transformation $\id_{\cA} \to \omega\ten_k \omega^{\vee}$ makes $C$ into a dg coalgebra over $k$.
Likewise,  $P$ becomes a right $C$-comodule and $Q$ a left $C$-comodule.

Also note that because $D$ is a cofibrant replacement for $\id_{\cA}$, we have
\[
C\simeq  \omega^{\vee}\ten_{\cA}^{\oL}\id_{\cA}\ten_{\cA}^{\oL}\omega \simeq \omega^{\vee}\ten^{\oL}_{\cA}\omega.
\]

%
%
%
%
%
%
%
 \subsubsection{Monoidal categories}\label{univbialg}
%
 For the purposes of this subsection $(\cA,\boxtimes,\mathbbm{1})$ is a strictly  monoidal dg category, so we have $k$-linear dg functors $\mathbbm{1} \co k \to \cA$ and  $\boxtimes \co \cA \ten \cA \to \cA$, such that if we also write $\mathbbm{1}$ for the image of the unique object in $k$,
\[
 (X\boxtimes Y)\boxtimes Z = X \boxtimes (Y \boxtimes Z), \quad \mathbbm{1} \boxtimes X = X, \quad X \boxtimes \mathbbm{1} = \mathbbm{1}.
\]

\begin{definition}
 Say that a dg functor $\omega\co \cA \to \per_{\dg}(k)$  is lax monoidal if it is equipped with natural transformations
\[
 \mu_{XY} \co \omega(X)\ten \omega(Y)\to \omega(X\boxtimes Y), \quad  \eta\co k \to \omega(\mathbbm{1})
\]
satisfying associativity and unitality conditions.

It is said to be strict (resp. strong, resp. quasi-strong) if $\mu$ and $\eta$ are equalities (resp. isomorphisms, resp. quasi-isomorphisms).
\end{definition}

By \cite[Proposition \ref{HHtannaka1-bialg}]{HHtannaka1}, if $\omega \co \cA \to \per_{\dg}(k) $ is strongly monoidal, the monoidal structures endow the   dg coalgebras $C_{\omega}( \cA), N C_{\omega}( \cA)$  with the natural structure of unital dg bialgebras. These are graded-commutative whenever $\boxtimes$ and $\omega$ are symmetric.

The monoidal structure $\boxtimes$ on $\cA$ induces a monoidal structure on $\cA^{\op}$, which we also denote by $\boxtimes$. There is also a monoidal structure $\boxtimes^2$ on $\cA^{\op}\ten \cA$, given by $(X\ten Y)\boxtimes^2(X'\ten Y'):= (X\boxtimes X')\ten (Y\boxtimes Y')$. As in \cite[Definition \ref{HHtannaka1-boxtimesdgdef}]{HHtannaka1}, these extend to dg functors
\begin{align*}
 \boxtimes \co \cD_{\dg}(\cA) \ten \cD_{\dg}(\cA) &\to \cD_{\dg}(\cA),\\
 \boxtimes \co \cD_{\dg}(\cA^{\op}) \ten \cD_{\dg}(\cA^{\op}) &\to \cD_{\dg}(\cA^{\op}),\\
\boxtimes^2 \co \cD_{\dg}(\cA^{\op} \ten \cA) \ten \cD_{\dg}(\cA^{\op} \ten \cA) &\to \cD_{\dg}(\cA^{\op} \ten \cA)
\end{align*}
 extending the dg functors to finite complexes, filtered colimits and direct summands.

\begin{definition}
 As in \cite[Definition \ref{HHtannaka1-univbialgdef}]{HHtannaka1}, we say that a universal coalgebra $D$ (in the sense of \S \ref{univcoalg}) is a universal bialgebra with respect to $\boxtimes$ if is equipped with an associative multiplication $D\boxtimes^2D \to D$ and a unit ${\mathbbm{1}}\ten {\mathbbm{1}} \to D$, both compatible with the coalgebra structure.
\end{definition}


As in \cite[Example \ref{HHtannaka1-HHunivbialg}]{HHtannaka1}, the Hochschild complexes of Example \ref{HHcoalgex}
are universal bialgebras whenever $k$ is a field.

%

By \cite[Lemma \ref{HHtannaka1-Calglemma}]{HHtannaka1}, for any universal bialgebra $D$ and a strong monoidal dg functor $\omega$, the dg coalgebra $C:= \omega^{\vee}\ten_{\cA}D\ten_{\cA}\omega$ becomes a unital associative dg bialgebra, which is commutative whenever $D$ is commutative and $\omega$  symmetric.

\begin{definition}
 Let $\C_{\dg}(C)$ be the dg category of right $C$-comodules in cochain complexes over $k$. Write $\C(C)$ for the underlying category $\z^0\C_{\dg}(C)$ of right $C$-comodules in cochain complexes, and $\cD(\C)$ for the homotopy category given by formally inverting quasi-isomorphisms. We then write $\cD_{\dg}(C)$ for  the full dg subcategory of $\C_{\dg}(C)$ on fibrant objects, for  model structure ``of the first kind'' described in   \cite[Remark 8.2]{positselskiDerivedCategories}.
\end{definition}

\subsubsection{Tannakian comparison}
%
%
%
\begin{definition}
 Write $\ker \omega$ for the full dg subcategory of  $\cD_{\dg}(\cA)$ consisting of objects $X$ with $\omega(X):=X\ten_{\cA}\omega$ quasi-isomorphic to $0$.   

Recall from \cite[\S 12.6]{drinfeldDGQuotient}  that the right orthogonal complement $(\ker\omega)^{\perp}\subset  \cD_{\dg}(\cA)$ is the full dg subcategory consisting  of those $X$ for which $\HHom_{\cA^{\op}}(M,X) \simeq 0$ for all $M \in \ker \omega$.
\end{definition}

The following is \cite[Theorem \ref{HHtannaka1-tannakathm}]{HHtannaka1}: 
\begin{theorem}\label{tannakathm}
For the constructions of $C \simeq \omega^{\vee}\ten_{\cA}^{\oL}\omega$ and the tilting module $P$  of \S \ref{tilt},
the derived adjunction $(-\ten_{\cA}P) \dashv \oR\HHom_C(P,-)$ gives rise to a quasi-equivalence between the dg categories $(\ker\omega)^{\perp}$ and  $\cD_{\dg}(C)$. Moreover, the map $(\ker\omega)^{\perp}\to \cD_{\dg}(\cA)/(\ker \omega)$  to the dg quotient is a quasi-equivalence.   
\end{theorem}

\section{Dense subcategories and semisimplicity}\label{SSsn}

The beauty of Theorem \ref{tannakathm} is that it describes the derived category $\cD(\cA)$ in terms of a fibre functor on $\cA$, so is invariant under Morita equivalences. In particular, for any derived Morita equivalence $\cB \to \cA$, we have a quasi-equivalence $\cD_{\dg}(\cB) \to \cD_{\dg}(\cA)$. This becomes particularly important when we can find a Morita equivalent dg category $\cB$ for which the category $\z^0\cB$ is semisimple, since the representing coalgebra then admits a particularly simple description. 

In \cite[Remark \ref{HHtannaka1-uniquenessrmk}]{HHtannaka1}, it was observed that different choices of universal coalgebra will give dg coalgebras which are derived Morita  equivalent. A quasi-isomorphism $C \to C'$ of dg coalgebras need not be a derived Morita equivalence, in general. 
However,  for the Tannakian dg coalgebras constructed in this section,  it turns out that quasi-isomorphisms will be derived Morita equivalences (see \S \ref{koszul} below).

\subsection{Non-negatively graded dg categories}\label{nonnegsn}

\begin{definition}\label{dgcoalgndef}
 Let $DG^{\ge 0}\Co_n\Alg_k$ denote the category of dg $k$-coalgebras $C$ in non-negative cochain degrees, satisfying the additional property that the map $\H^0C \to C$ of coalgebras is ind-conilpotent. This means that we can write $C$ as a nested union $C= \LLim_{\alpha} C_{\alpha}$ of dg coalgebras with $\H^0C= \H^0C_{\alpha}$ for all $\alpha$ and $C_{\alpha}/\H^0C$ conilpotent in the sense that the comultiplication
\[
 C_{\alpha}/\H^0C \to (C_{\alpha}/\H^0C)^{\ten m}
\]
is $0$ for some $m\ge 2$.
\end{definition}

For any $C \in DG^{\ge 0}\Co_n\Alg_k$, the maximal cosemisimple subcoalgebra (\cite[4.3]{positselskiDerivedCategories}) $C_{\red}:=(\H^0C)_{\red}\subset \H^0C$ thus gives an ind-conilpotent map $C_{\red}\to C$. Since $C_{\red}$ is cosemisimple, the ind-conilpotent morphism $C_{\red} \to C_0$ admits a retraction, so $C $ is of the form $C= C_{\red}\oplus N$, for $N$ an ind-conilpotent dg coalgebra with a compatible $C_{\red}$-bicomodule structure.

\begin{proposition}\label{coalgconc1}
Take a $k$-linear  dg category $\cA$ with  $\cA(X,Y)$ concentrated in non-negative degrees, $d\cA^0(X,Y)=0$ for all $X,Y$, and with 
$\cA^0$ a semisimple abelian category. Assume that we have a $k$-linear functor $\omega \co \cA^0 \to \FD\Vect_k$. Then there is a model for the  coalgebra $C\simeq \omega^{\vee}\ten_{\cA}^{\oL}\omega$ of \S \ref{tilt} with $C \in DG^{\ge 0}\Co_n\Alg_k$.
\end{proposition}
\begin{proof}
For $i \co \cA^0 \into \cA$, we set $D$ to be the direct sum total complex $N\CCC(\cA/\cA^0, i^{\op}\ten i) $ of the normalisation $N\CCC_{\bt}(\cA/\cA^0,i^{\op}\ten i )$ of the simplicial cochain complex $\uline{\CCC}_{\bt}(\cA/\cA^0,i^{\op}\ten i )$ given by 
\[
 \uline{\CCC}_n(\cA/\cA^0,i^{\op}\ten i ):= \cA(-,-)\ten_{\cA^0}\underbrace{\cA(-,-)\ten_{\cA^0}\ldots \ten_{\cA^0}\cA(-,-)}_n\ten_{\cA^0}\cA(-,-).
\]
Equivalently, $N\CCC_{\bt}(\cA/\cA^0,i^{\op}\ten i )$  is the total complex of 
\[
 n \mapsto \cA(-,-)\ten_{\cA^0}\underbrace{\cA^{>0}(-,-)\ten_{\cA^0}\ldots \ten_{\cA^0}\cA^{>0}(-,-)}_n\ten_{\cA^0}\cA(-,-).
\]

The comultiplication and counit are  given by the  formulae of \cite[Proposition \ref{HHtannaka1-coalg}]{HHtannaka1}, so we need only show that the counit $D \to \id_{\cA}$ is a quasi-isomorphism and that $D$ is a cofibrant module.

The identity $\id_X \in \cA(X,X)$ gives a contracting homotopy of the complex $D(X,Y) \to \cA(X,Y)$, which ensures that the counit is a quasi-isomorphism. To see that $D$ is cofibrant, we just note that
\[
 \cA(-,iX)\ten_k\underbrace{\cA^{>0}(-,-)\ten_{\cA^0}\ldots \ten_{\cA^0}\cA^{>0}(-,-)}_n\ten_k\cA(iY,-)
\]
is a cofibrant module for all $X,Y \in \cA^0$, and that semisimplicity of $\cA^0$ ensures that taking $\cA^0$-coinvariants is an exact functor, so preserves cofibrancy.

Now, observe that $C:= \omega^{\vee}\ten_{\cA}D\ten_{\cA}\omega$ is the direct sum total complex of
\[
n \mapsto \omega^{\vee}\ten_{\cA^0}\underbrace{\cA^{>0}(-,-)\ten_{\cA^0}\ldots \ten_{\cA^0}\cA^{>0}(-,-)}_n\ten_{\cA^0}\omega, 
\]
which has no negative terms, since $\omega$ is concentrated in degree $0$ and $(\cA^{>0})^{\ten n}$ in degrees $\ge n$.

Finally, observe that the morphism $\omega^{\vee}\ten_{\cA^0}\omega \to C$ is an ind-conilpotent extension (cofiltering  by copowers of $\omega^{\vee}\ten_{\cA^0}\cA^{>0}\ten_{\cA^0}\omega $) and that $\omega^{\vee}\ten_{\cA^0}\omega \subset \H^0C$, so the morphism $\H^0C \to C$ is necessarily also conilpotent.
\end{proof}

\begin{remark}\label{cechnerve}
 Like the construction of Example \ref{HHcoalgex}, the universal coalgebra of Proposition \ref{coalgconc1} can be written as a \v Cech nerve. Set $L:= \cA(-,-)\ten_{\cA^0}\cA(-,-)$, which is a $\ten_{\cA}$-coalgebra in $\C_{\dg}(\cA^{\op}\ten \cA)$. We then have
 \[
        \uline{\CCC}_n(\cA/\cA^0,i^{\op}\ten i ) =\underbrace{L\ten_{\cA}L \ten_{\cA} \ldots \ten_{\cA}L}_{n+1},
 \]
giving the \v Cech nerve of the $\ten_{\cA}$-comonoid $L$. 
\end{remark}

\begin{remark}\label{irredHHrmk}
 If $k$ is algebraically closed, then the  complex $ CC_{\bt}(\cA/\cA^0,i^{\op}\ten i )$ admits a simpler description. Let $\{V_{\alpha}\}_{\alpha}$ be a set of irreducible objects of $\cA^0$, with one in each isomorphism class. Since $k$ is algebraically closed, $\End_{\cA^0}(V_{\alpha})\cong k$, and we get
 \begin{align*}
&\uline{\CCC}_n(\cA/\cA^0,i^{\op}\ten i )\cong \\
&\bigoplus_{\alpha_0, \ldots, \alpha_n}  \cA(-,V_{\alpha_0})\ten_k\cA(V_{\alpha_0},V_{\alpha_1})\ten_k\ldots \ten_k\cA(V_{\alpha_{n-1}},V_{\alpha_n})\ten_k\cA(V_{\alpha_n},-).     
 \end{align*}

Writing $\cA_s \subset \cA$ for the full dg subcategory on objects $\{V_{\alpha}\}_{\alpha}$, this gives an isomorphism
\[
\CCC_{\bt}(\cA_s,i^{\op}\ten i )  \cong \CCC_{\bt}(\cA/\cA^0,i^{\op}\ten i ).     
\]
Thus the quasi-isomorphism  $\CCC_{\bt}(\cA/\cA^0,i^{\op}\ten i ) \to\CCC_{\bt}(\cA,h_{\cA^{\op}}\ten h_{\cA} ) $ is a consequence of the derived Morita equivalence $\cA_s \to \cA$.
\end{remark}

\begin{corollary}\label{coalgconc2}
 Take a $k$-linear dg category $\cA$ with  $\cA(X,Y)$ concentrated in non-negative degrees, and
$\H^0\cA$  a semisimple abelian category. Assume that we have a $k$-linear functor $\omega \co \cA^0 \to \FD\Vect_k$. Then there is a dg coalgebra $C\in DG^{\ge 0}\Co_n\Alg_k$ with $C\simeq \omega^{\vee}\ten_{\cA}^{\oL}\omega$,   together with  quasi-equivalences  $\cD_{\dg}(\cA)/(\ker \omega) \simeq  (\ker \omega)^{\perp} \simeq \cD_{\dg}(C)$.
\end{corollary}
\begin{proof}
Consider the morphism $d\co \cA^0(-,-) \to \cA^1(-,-)$ of $\H^0\cA$-bimodules. Since $\H^0\cA$ is semisimple, there exists a $\H^0\cA$-bimodule decomposition
\[
 \cA^1(-,-) = d\cA^0(-,-)\oplus \cB^1(-,-).
\]
We may therefore define a dg subcategory $\cB \subset \cA$ by
\[
 \cB^n(-,-):= \left\{\begin{matrix}
                      \cA^n(-,-)& n\ne 0,1 \\
\cB^1(-,-) & n=1\\
\H^0\cA(-,-) & n=0.
               \end{matrix}\right.
\]
Then $\cB \to \cA$ is a quasi-equivalence, and $\cB$ satisfies the conditions of Proposition \ref{coalgconc1}, giving a dg coalgebra $C$  concentrated in non-negative degrees. We then apply Theorem \ref{tannakathm}.
\end{proof}


\begin{definition}\label{D+dgdef}
 For a dg coalgebra $C \in DG^{\ge 0}\Co_n\Alg_k$ , define $\cD^+_{\dg}(C) \subset \cD_{\dg}(C)$ to be the full dg subcategory on cochain complexes $V$ for which $\H^*(V)$ is bounded below.  Write $\cD^+(C):= \H^0\cD^+_{\dg}(C)$

For a $k$-linear  dg category $\cA$ with  $\cA(X,Y)$ concentrated in non-negative degrees, define $\cD_{\dg}^+(\cA) \subset \cD_{\dg}(\cA)$ to be the full  dg subcategory consisting of dg functors $F$ for which $\prod_{X \in \cA}\H^*F(X)$ is bounded below. Write $\cD^+(\cA):= \H^0\cD_{\dg}^+(\cA)$. 
\end{definition}

\begin{proposition}\label{coalgconcequiv}
 Under the conditions of Corollary \ref{coalgconc2}, if the functor $\omega \co \H^0\cA \to \FD\Vect_k$  is faithful, then we have a quasi-equivalence $\cD_{\dg}^+(\cA) \simeq \cD_{\dg}^+(C)$.
\end{proposition}
\begin{proof}
We first replace $\cA$ with the dg category $\cB$ from the proof of Corollary \ref{coalgconc2}, so $\cB^0=\H^0\cB$.
  Since $\omega$ is additive and $\cB^0$ semisimple, it follows that $\omega|_{\cB^0}$ is exact, and hence represented by some $T\in \ind((\cB^0)^{\op})$, with $\omega|_{\cB^0} = -\ten_{\cB^0}T$. We may write $T$ as a filtered colimit $T= \LLim T_{\alpha}$ for $T_{\alpha} \in (\cB^0)^{\op}$, and because $\cB^0$ is abelian we may assume that each $T_{\alpha}$ is a subobject of $T$. Since $\cB^0$ is semisimple, this means that $T_{\alpha}$ is a direct summand of $T$. 

Because  $\omega|_{\cB^0}$ is faithful, it follows that the set $\{T_{\alpha}\}_{\alpha}$ generates $\cB^0$.   Now, the dg functor $\omega\co \cD_{\dg}(\cB) \to \C_{\dg}$ is $(-\ten_{\cB}\cB^0(-,-)\ten_{\cB^0}T)$, 
so if $\omega(M) \simeq 0$ then 
\[
 M \ten_{\cB}\cB^0(-,T_{\alpha}) =M\ten_{\cB}\cB^0(-,-)(T_{\alpha}) \simeq 0
\]
 for all $\alpha$ ($T_{\alpha}$ being a direct summand of $T$). 
Since $\{T_{\alpha}\}_{\alpha}$ generates $\cB^0$, it follows that  $M \ten_{\cB}\cB^0(-,-)(X) \simeq 0$ for all objects $X \in \cB$, which is precisely the same as saying that $M \ten_{\cB}\cB^0(-,-)\simeq 0$.

For the natural projection $\pi \co \cB \to \cB^0$, this says that $\pi^*M\simeq 0$;  for any $\cB^0$-module $N$, we then have
$
 \HHom_{\cB}(M, \pi_*N) \simeq 0,
$
so any complex quasi-isomorphic to $\pi_*N$ lies in  $ (\ker \omega)^{\perp}$. Since $(\ker \omega)^{\perp}$ is closed under extensions and homotopy limits, and any $R \in \cD_{\dg}^+(\cB)$ can be recovered from the $\cB^0$-modules $\H^iR$ via these operations, it  follows that  $ R \in (\ker \omega)^{\perp}$.

Thus Corollary \ref{coalgconc2}  shows that the dg functor $-\ten_{\cB}P$ gives a quasi-equivalence from  $\cD^+_{\dg}(\cB)$ to a full dg subcategory of $\cD_{\dg}^+(C)$. It remains to show that for any $N \in \cD_{\dg}^+(C)$, we have $\HHom_C(P, N) \in \cD^+_{\dg}(\cB)$; without loss of generality we may assume $\H^{<0}N=0$. Now, 
\[
 \prod_{X \in \cB} \HHom_C(P, N)(X)=  \prod \HHom_C(X\ten_{\cB}P, N),
\]
 and $ \H^*(X\ten_{\cB}P) \cong \H^*(\omega X)$, which is concentrated in degree $0$. By applying  the cobar resolution in the  proof of \cite[Theorem 4.4]{positselskiDerivedCategories}  to $\tau^{\ge 0}N$, it follows that that $N$ is quasi-isomorphic to a $C$-comodule $N'$  concentrated in non-negative degrees and fibrant in the coderived model structure. Then \cite[Theorem 4.3.1]{positselskiDerivedCategories}  implies that 
$\HHom_C(-, N)\simeq \HHom_C(-, N')$, so  
\[
 \HHom_C(X\ten_{\cB}P, N)\simeq \HHom_C(\H^0(X\ten_{\cB}P), N'),
\]
which
is concentrated in non-negative degrees.
\end{proof}

\begin{definition}\label{Dcodef}
For any dg coalgebra $C$, define $\cD_{\dg}^{\mathrm{co}}(C)$ to be the full dg subcategory of $\C_{\dg}(C)$ on  objects $K$ for which the graded module $K^{\#}$ underlying $K$ is injective as a comodule over the graded coalgebra $C^{\#}$ underlying $C$. 
\end{definition}
Note that from the properties of the model structure of \cite[Theorem 8.2]{positselskiDerivedCategories}, the homotopy category $\H^0\cD_{\dg}^{\mathrm{co}}(C)$ is equivalent to  Positselski's coderived category $\cD^{\mathrm{co}}(C)$. Weak equivalences with respect to this model structure are morphisms whose cone $L$ is coacyclic in the sense of \cite[4.2]{positselskiDerivedCategories} --- this is a stronger condition than acyclicity, and is equivalent to saying that $\HHom_C(L,K)$ is acyclic for all $K \in \cD_{\dg}^{\mathrm{co}}(C)$.

\begin{proposition}\label{coalgconcper}
The equivalence of Proposition \ref{coalgconcequiv} induces a quasi-equivalence between   $\per_{\dg}(\cA)$ (see Definition \ref{perdef}) and the full dg subcategory $\cF_{\dg}(C)$ of  $\cD_{\dg}(C)$ 
on fibrant replacements of $C$-comodules in  finite-dimensional cochain complexes. This gives a quasi-equivalence from $\cD_{\dg}(\cA) $ to $ \cD_{\dg}^{\mathrm{co}}(C)$.
\end{proposition}
\begin{proof}
Observe that when filtered colimits exist in $\cD_{\dg}^+(\cA)$, they are homotopy colimits, and that $\Hom_{\cA}(K,-)$ commutes with such limits for all $K \in \per_{\dg}(\cA)$. Since every object of $\cD_{\dg}(\cA)$ can be written as a  filtered colimit of perfect complexes, it follows that $\per_{\dg}(\cA)$ consists of the homotopy-compact objects of $\cD_{\dg}^+(\cA)$ (i.e. the objects $K$ for which $\Hom_{\cD^*}(\cA)(K,-)$ commutes with filtered homotopy colimits, when they exist). They therefore correspond under Proposition \ref{coalgconcequiv} to the homotopy-compact objects of $\cD^+_{\dg}(C)$. 

By \cite[Theorem 4.3.1(a)]{positselskiDerivedCategories}, $\cD^+_{\dg}(C) $   is quasi-equivalent to the full subcategory of $\cD_{\dg}^{\mathrm{co}}(C) $ consisting of cochain complexes which are bounded below. By  \cite[5.5]{positselskiDerivedCategories},    $C$-comodules in  finite-dimensional cochain complexes are compact generators of $D^{\mathrm{co}}(C) $, and hence of $\cD^+(C)$. Thus the essential image of 
\[
 -\ten_{\cA}P \co \per_{\dg}(\cA) \to \cD_{\dg}(C)
\]
is just $ \cF_{\dg}(C)$.

Moreover, since the finite-$k$-dimensional $C$-comodules generate $ \cD^{\mathrm{co}}(C)$ and the latter is closed under arbitrary direct sums, we have  
a quasi-equivalence
\[
 \LLim \co \ind(\cF_{\dg}(C)) \to \cD_{\dg}^{\mathrm{co}}(C),
\]
which combines with the quasi-equivalence $\per_{\dg}(\cA) \to\cF_{\dg}(C)$ above to give a quasi-equivalence $\cD_{\dg}(\cA) \to \cD_{\dg}^{\mathrm{co}}(C) $ on the associated ind-categories.
\end{proof}

\begin{example}
Observe that if $ \ker \omega|_{\H^0(\cA)}=0 $, we  need not have $\ker \omega =0$ on $ \cD_{\dg}(\cA)$. 
We see this by considering an example which is in some respects dual to \cite[Example \ref{HHtannaka1-kellerex}]{HHtannaka1}.
For $t$ of degree $1$,  we can take $k\<t\>$ to be the free non-commutative graded algebra generated by $t$ and let $\cA$  be the full dg subcategory of $\cD_{\dg}(k\<t\>)$ on objects $k\<t\>^n$. Then $\cD_{\dg}(\cA) \simeq \cD_{\dg}(k\<t\>)$, and the dg fibre functor $\omega(M):= k\ten_{ k\<t\>}M$ is faithful on $\cD^+(\cA) $. However, it is not faithful on $\cD(\cA)$, since $k\<t,t^{-1}\>$ lies in the kernel.

The associated dg coalgebra $C$ is Morita equivalent to $k[\eps]^{\vee}$, for $\eps$ of degree $0$ with $\eps^2=0$, so
 a Corollary \ref{coalgconc2} in this case gives
\[
 \cD_{\dg}(k\<t\>)/( k\<t,t^{-1}\>) \simeq \cD_{\dg}(k[\eps]),
\]
while Proposition \ref{coalgconcper} gives an equivalence between $\per_{\dg}(k\<t\>)$ and the dg derived category of finite $k[\eps]$-modules.  The difference between derived and coderived categories in this case can also be seen by noting that the $k[\eps]$-module $k$ is not perfect, but is compact in the coderived category.
\end{example}

\subsection{Koszul duality}\label{koszul}

The correspondence of Proposition \ref{coalgconcper} is a manifestation of Koszul duality between modules and comodules, and can be regarded as a partial generalisation of \cite[Theorem 6.3.a]{positselskiDerivedCategories}. 
In particular, $\cA \leadsto C$ is a cobar construction and $-\ten_{\cA}P$ can be thought of as $\omega\ten^{\tau,S}C$ for the canonical twisting cochain $\tau$.

Rather than fixing a dg category, we now use Koszul duality  to give an equivalence between certain homotopy categories of dg categories and of dg coalgebras. A consequence is that quasi-isomorphisms in the category $DG^{\ge 0}\Co_n\Alg$ (see Definition \ref{dgcoalgndef}) induce quasi-isomorphisms of the associated categories, so are necessarily derived Morita equivalences.

Fix a cosemisimple coalgebra $S$, and  set $\cS$ to be the category of finite-dimensional $S$-comodules, with $\omega$ the forgetful functor to vector spaces. We can then interpret $\cS$-bimodules as $S$-bicomodules, observing that $S= \omega^{\vee}\ten_{\cS}\omega$. Note that a  non-counital coassociative dg $\ten_{\cS}$-coalgebra $\cB$ in $\cS$-bimodules  then corresponds to the non-counital coassociative dg $\ten^S$-coalgebra $B:=\omega^{\vee}\ten_{\cS}\cB\ten_{\cS}\omega$ in $S$-bicomodules. This is equivalent to a coassociative
dg coalgebra structure on $S \oplus B$, 
for which $S \to S \oplus B \to S$ are morphisms of dg coalgebras.

\begin{definition}
Let  $DG^{> 0}\Cat(\cS)$ be the category of dg categories $\cA$ in non-negative degrees with $\cA^0=\cS$ and $d\cA^0=0$, as considered in Proposition \ref{coalgconc1}.
\end{definition}
This is equivalent to the category of associative  $\ten_{\cS}$-algebras $\cA^{>0}(-,-)$ in cochain complexes of $\cS$-bimodules in strictly positive degrees. Applying \cite[Theorem 11.3.2]{Hirschhorn} to the forgetful functor mapping to $\cS$-bimodules,  it follows that  $DG^{> 0}\Cat(\cS)$ has a cofibrantly generated model structure in which weak equivalences are quasi-isomorphisms and fibrations are surjections.

\begin{definition}
 Dually, let $DG^{{\ge 0}}\Co_n\Alg(\cS)$ be the category of non-counital ind-conilpotent coassociative dg $\ten_{\cS}$-coalgebras $B$ in complexes of $\cS$-bimodules in non-negative cochain degrees. 
\end{definition}
By analogy with \cite[Theorem 3.1]{hinstack} and  \cite[Proposition 1.26]{ddt1}, $DG^{{\ge 0}}\Co_n\Alg(\cS)$ has a fibrantly cogenerated model structure in which weak equivalences are quasi-isomorphisms and cofibrations are injective in degrees $>0$.

\begin{definition}\label{betadef}
Write $\beta(\cA)$ for the cofree ind-conilpotent graded $\ten_{\cS}$-coalgebra on generators $\cA^{>0}(-,-)[1]$; thus 
\[
 \beta(\cA)=\bigoplus_{n>0} \underbrace{\cA^{>0}(-,-)\ten_{\cS}\ldots \ten_{\cS}\cA^{>0}(-,-)}_n[n].
\]
We make this a dg coalgebra by defining the differential on cogenerators to be 
\[
 d_{\beta(\cA)}= (d_{\cA}, \circ) \co (\cA^{>0}(-,-)[1]) \oplus (\cA^{>0}(-,-)\ten_{\cS}\cA^{>0}(-,-)[2]) \to \cA^{>0}(-,-)[2];
\]
\end{definition}
Note that the dg coalgebra $C$ of Proposition \ref{coalgconc1} is just the dg coalgebra $S \oplus (\omega^{\vee}\ten_{\cS}\beta(\cA)\ten_{\cS}\omega)$.
This cobar construction defines a functor $\beta \co DG^{> 0}\Cat(\cS) \to DG^{{\ge 0}}\Co_n\Alg(\cS)$, with  Proposition \ref{coalgconc1} saying that $\cD_{\dg}(\cA) \simeq \cD_{\dg}( S\oplus (\omega^{\vee}\ten_{\cS}\beta(\cA)\ten_{\cS}\omega))$.

\begin{definition}\label{beta*def}
 Let  $\beta^*$ be the left adjoint to $\beta$. This is the bar construction sending $C$ to the tensor algebra 
\[
 \beta^*(C)(X,Y) = \bigoplus_{n\ge 0}\underbrace{C(X,-)\ten_{\cS}\ldots \ten_{cS}C(-,Y)}_n[-n],
\]
with differential defined on generators by $d_C + \Delta_C$.
\end{definition}

\begin{remark}
 A key observation is that the filtration of  $\beta(C)$ by powers of $C[-1]$ gives a convergent spectral sequence
\[
 \H^q(C^{\ten p}) \abuts \H^{p+q}(\beta^*C).
\]
Note that convergence of this spectral sequence relies on $\H^{<0}(C)$ vanishing, and allows us to use quasi-isomorphisms for our notion of dg coalgebra weak equivalences where \cite[Theorem 3.1]{hinstack} used $\beta^*$ to reflect weak equivalences.
\end{remark}

\begin{proposition}\label{Koszulequiv}
 The functors $\beta^*\dashv \beta$ are a pair of Quillen equivalences between the categories $DG^{> 0}\Cat(\cS), DG^{{\ge 0}}\Co_n\Alg(\cS)$.
\end{proposition}
\begin{proof}
 We need to show that for any $C \in DG^{{\ge 0}}\Co_n\Alg(\cS)$, the unit $C \to \beta\beta^*C$ of the adjunction is a quasi-isomorphism, and that for any $\cA \in DG^{> 0}\Cat(\cS)$, the co-unit $\beta^*\beta \cA \to \cA$ is a quasi-isomorphism.

We begin by noting that \cite[Proposition 11.4.4]{lodayvalletteoperads} says that for $C$ fibrant, the unit $ C \to \beta\beta^*(C) $ gives a  quasi-isomorphism on tangent spaces, where $\tan(C) = \ker(\Delta\co C \to C\ten C)$.  Now, $\tan\beta(\cA)= \cA^{>0}(-,-)$, so setting $C= \beta(\cA)$, it follows that the unit $ C \to \beta\beta^*(C) $ gives a quasi-isomorphism $\cA^{>0} \to \beta^*\beta(\cA)^{>0}$, from which it follows that the co-unit is a quasi-isomorphism. 

Now for $C$ fibrant, filtration by the subspaces $\ker(C \to C^{\ten n})$ gives a convergent spectral sequence
\[
 E_1^{pq}= \H^{p+q}((\tan C)^{\ten -p} ) \abuts \H^{p+q}(C),
\]
so cotangent quasi-isomorphisms of fibrant objects are always quasi-isomorphisms, and in particular the unit $ C \to \beta\beta^*(C) $ is a  quasi-isomorphism for fibrant $C$.  Since the functor $\beta\beta^*$ preserves quasi-isomorphisms, the unit must be a quasi-isomorphism for all $C$.
\end{proof}

\begin{remark}\label{Koszulremark}
 The key step in  Proposition \ref{Koszulequiv} invokes Koszul duality in the form of \cite[Proposition 11.4.4]{lodayvalletteoperads}. When the field $k$ has characteristic $0$, there are therefore analogues for any Koszul-dual pair of operads, with cochain dg $\cP$-algebras in strictly positive degrees corresponding to conilpotent $\cP^{\text{\textexclamdown}}$-coalgebras in non-negative degrees. 
\end{remark}

\subsection{Hearts of $t$-structures}\label{heartsn}

For a Morita fibrant dg category $\cD$ to admit a compatible $t$-structure amounts to the existence of a full generating dg subcategory $\cA$  with  $\H^i\cA(X,Y)=0$ for all $i<0$ and all $X,Y$. The objects of $\cA$ are given by any choice of generators for the heart $\H^0\cD^{\heartsuit}$ of the $t$-structure, and in particular we can take $\cA$ to be the full dg subcategory of $\cD$ on the semisimple objects of $\H^0\cD^{\heartsuit}$, in which case $\H^0\cA$ will be abelian semisimple.

We now show how to extend  the results of \S \ref{nonnegsn} to this generality. 

\begin{proposition}\label{filterprop}
 Take a $k$-linear  dg category $\cA$ with the category $\H^0\cA$ abelian semisimple and $\H^{<0}\cA(X,Y)=0$ for all objects $X,Y$. Then $\cA$ is quasi-isomorphic to a dg category $\cB$ concentrated in non-negative degrees, with $d\cB^0(X,Y)=0$ for all $X,Y$.
\end{proposition}
\begin{proof}
First, observe that the good truncation filtration $\tau_n= \tau^{\le n}$ for $n \ge 0$ gives quasi-isomorphisms
\[
 \bigoplus_n \gr^{\tau}_n\cA(X,Y) 
 \to    \bigoplus_n\H^n\cA(X,Y)[-n] 
\]
for all $X,Y \in \cA$, giving $\bG_m$-equivariant quasi-isomorphisms of the corresponding dg categories, where $\gr^{\tau}_n$ is assigned weight $n$. 

Now, a dg category $\cB$ over $\tau_0\cA$ on the same objects corresponds to the associative unital $\ten_{\tau_0\cA}$-algebra $\cB(-,-)$ in $\C_{\dg}((\tau_0\cA)^{\op}\ten \tau_0\cA)$. 
Thus the  quasi-isomorphism $\tau_0\cA  \to\H^0\cA$  
ensures that $\cA$ is quasi-isomorphic to some dg category $\cA'$ over $\H^0A$.

Consider the polynomial ring $k[t]$ with $t$ in degree $0$ but equipped with a $\bG_m$-action of weight $1$.
The Rees construction gives us a dg category $\zeta(\cA,\tau)$ with the same objects as $\cA$ and morphisms
\[
 \zeta(\cA,\tau)(X,Y):= \bigoplus_n \tau_n\cA(X,Y)t^n\cong\bigoplus_n \tau_n\cA(X,Y) ,
\]
which then becomes a $\bG_m$-equivariant dg category, flat over $k[t]$. This has the properties that $\zeta(\cA,\tau)/t\cong \gr^\tau\cA$ and $\zeta(\cA,\tau)/(t-1)\cong \cA$. 

Applying the argument above, we see that the $\bG_m$-equivariant dg category $\zeta(\cA,\tau)$ over $\tau_0\cA[t]$ must be  quasi-isomorphic to some $\bG_m$-equivariant dg category $\cZ$ over $\H^0\cA[t]$ (flat over $k[t]$). Note that $\cZ/t$ and $\cZ/(t-1)$ are then quasi-isomorphic to $\gr^\tau\cA$ and $\cA$, respectively. 

Now, set  $\cG^0= \H^0\cA$, then take a cofibrant replacement $\cG^{>0}(-,-)$ of $\H^{>0}\cA(-,-)$ as a $\ten_{\cG^0}$-algebra in $\cG$-bimodules. This can be constructed canonically as a bar-cobar resolution as in Proposition \ref{Koszulequiv}, with the output  concentrated in strictly positive degrees (this relies on the semisimplicity of $\H^0\cA$). Moreover, the $\bG_m$-action on $\H^*(\cA(X,Y)) $  (with $\H^n$ of weight $n$) transfers equivariantly to $\cG:= \cG^0 \oplus \cG^{>0}$.

Since $\cG$ is cofibrant and $\gr^\tau\cA$ is quasi-isomorphic to $\H^*\cA$, we may lift the quasi-isomorphism $\cG \to \H^*\cA$ to give a $\bG_m$-equivariant quasi-isomorphism
\[
 \cG \to \cZ/t.
\]

We now mimic \cite[Propositions \ref{mhs2-relalghgs},\ref{mhs2-strictlift}]{mhs2}.  
Since $\cG(-,-)$ is cofibrant as a unital $\ten_{\cG^0}$-algebra, forgetting the differential gives a retract $\cG^*$ of a freely generated $\bG_m$-equivariant $\ten_{\cG^0}$-algebra. Since $\cZ \to \cZ/t$ is surjective, we may lift the map $\cG \to \cZ/t$ to give a $\bG_m$-equivariant map $f\co \cG^* \to \cZ$ of graded categories, and hence $\cG^*[t] \to \cZ$.

We then consider possible differentials on $\cG^{>0}[t]$ making $f$ into a map of $\ten_{\cG^0}$-algebras.   The proof of \cite[Proposition \ref{dmsch-alghgs}]{dmsch} characterises obstructions to passing from a differential on  $\cG^{>0}[t]/t^r$ to one on $\cG^{>0}[t]/t^{r+1}$ in terms of elements in Hochschild cohomology, 
which necessarily vanish because the lift $\cZ/t^{r+1} \to \cZ/t^r $ exists. In the $\bG_m$-equivariant category, $\cG^*[t]$ is the limit $\Lim_r \cG^*[t]/t^r$, so a suitable differential $\delta$ exists on $\cG^*[t]$, set to $0$ on $\cG^0[t]$.

Writing $\cR=(\cG^*[t], \delta)$ with its differential, we thus have a $\bG_m$-equivariant quasi-isomorphism
\[
 \cR \to \cZ
\]
over $k[t]$ (deformations of quasi-isomorphisms being quasi-isomorphisms), and hence a quasi-isomorphism
\[
 \cR/(t-1)\to \cZ/(t-1) \simeq \cA,
\]
so $\cB:= \cR/(t-1)$ has the required properties.
\end{proof}

\begin{corollary}\label{coalgconc3}
 Take a $k$-linear  dg category $\cA$ with the category $\H^0\cA$ abelian semisimple and $\H^{<0}\cA(X,Y)=0$ for all objects $X,Y$. Assume that we have a $k$-linear dg functor $\omega \co \cA \to \C_{\dg}(k)$ with $\H^*\omega(X)$  finite-dimensional and concentrated in degree $0$ for all $X \in \cA$. Then there is a dg coalgebra $C \in DG^{\ge 0}\Co_n\Alg_k$ with $ C\simeq \omega^{\vee}\ten_{\cA}^{\oL}\omega$,   together with  quasi-equivalences  $\cD_{\dg}(\cA)/(\ker \omega) \simeq  (\ker \omega)^{\perp} \simeq \cD_{\dg}(C)$.
\end{corollary}
\begin{proof}
 By Proposition \ref{filterprop}, there is a quasi-isomorphism $q \co\cB \to \cA$ with $\cB$ satisfying the conditions of Proposition \ref{coalgconc1}. It therefore suffices to replace $\omega \circ q\co \cB \to \C_{\dg}(k)$ with a quasi-isomorphic dg functor taking values in $\FD\Vect_k$. Now, since $\cB^0 \simeq \H^0\cA$ is semisimple, we may decompose the $\cB^0$-module $(\omega\circ q)^0$ as $d(\omega\circ q)^{-1} \oplus M^0$ for some  $\cB^0$-module $M^0$. Setting $M^i = (\omega\circ q)^i$ for $i>0$ and $M^{<0}=0$ then gives a quasi-isomorphism $M \to \omega\circ q$ of $\cB$-modules. Next, choose a decomposition $M^0= \z^0M \oplus N^0$ of $\cB^0$-modules, and set $N^i=M^i$ for $i>0$. Since $\H^iM=0$ for $i\ne 0$, it follows that $N$ is acyclic, so $M \to M/N$ is a quasi-isomorphism, and $M/N\cong \H^0(\omega\circ q)$ takes values in $\FD\Vect_k$.  
\end{proof}

\begin{example}[The motivic coalgebra and mixed motives]\label{mothom2}
As in \cite[\S \ref{HHtannaka1-mothomsn}]{HHtannaka1}, consider the  dg functor $E^{\vee} \co \cM_{\dg,\bA^1,c}(F,\Q) \to \C_{\dg}(\Q)$  associated to a mixed Weil homology theory, and let $C$ be the associated dg coalgebra, with quasi-equivalences
\[
 \cD_{\dg}(C) \simeq \cM_{\dg,\bA^1}(F,\Q)/(\ker E) \simeq \cM_{\dg}(F,\Q)/(\ker E).
\]
Similarly, let $C^{\eff}$ be the dg coalgebra associated to  the the restriction of $E$ to effective Beilinson motives, giving
\[
 \cD_{\dg}(C^{\eff}) \simeq \cM_{\dg,\bA^1}^{\eff}(F,\Q)/(\ker E) \simeq \cM_{\dg}^{\eff}(F,\Q)/(\ker E),
\]
where $\cM_{\dg}^{\eff}(F,\Q)$ is the dg category of  cofibrant presheaves of $\Q$-complexes on smooth $F$-schemes.

When $E$ is Betti cohomology, it is shown in \cite[Corollary 2.105]{ayoubGaloisMotivic1}  that $\H^{>0}C=0$ and $\H^{>0}C^{\eff}=0$. By \cite[Lemma 2.145]{ayoubGaloisMotivic1}, existence of a motivic $t$-structure would also imply that $\H^{<0}C=0$, so we would have quasi-isomorphisms $C \to \tau^{\ge 0}C \la \H^0C$ of dg coalgebras, but it is not immediate that these are Morita equivalences.  

However, if a motivic $t$-structure exists, then  $\cM_{\dg,\bA^1}(F,\Q) \subset (\ker E)^{\perp}$
and
applying Corollary \ref{coalgconc3} to the full dg subcategory $\cA$ of $\cM_{\dg,\bA^1,c}(F,\Q)$ or $\cM_{\dg,\bA^1,c}^{\eff}(F,\Q)$ on semisimple objects in the heart of the $t$-structure would yield $N$ or $N^{\eff}$ in $DG^{\ge 0}\Co_n\Alg_k$, Morita equivalent to $C$ or $C^{\eff}$, so by Propositions \ref{coalgconcequiv} and \ref{coalgconcper},
\begin{align*}
 \cM_{\dg,\bA^1}(F,\Q)\simeq \cD_{\dg}^{\mathrm{co}}(N), \quad \cM_{\dg,\bA^1}^{\eff}(F,\Q)\simeq \cD_{\dg}^{\mathrm{co}}(N^{\eff}),\\
\cM_{\dg,\bA^1}^+(F,\Q)\simeq \cD_{\dg}^+(N), \quad \cM_{\dg,\bA^1}^{\eff,+}(F,\Q)\simeq \cD_{\dg}^+(N^{\eff}),
\end{align*}
where $ \cM_{\dg,\bA^1}^+(F,\Q)$ is the full dg subcategory of $\cM_{\dg,\bA^1}(F,\Q) $ consisting of objects in $\bigcup_n \cM_{\bA^1}(F,\Q)^{\ge n}$ for the motivic $t$-structure.

By \cite[Corollary 2.105]{ayoubGaloisMotivic1}, the morphisms $\H^0N \to N$ and $\H^0N^{\eff} \to N^{\eff}$ would then be  quasi-isomorphisms, hence  Morita equivalences by Proposition \ref{Koszulequiv}, and we would have
\[
 \cM_{\dg,\bA^1}(F,\Q)\simeq \cD_{\dg}^{\mathrm{co}}(\H^0N), \quad \cM_{\dg,\bA^1}^{\eff}(F,\Q)\simeq \cD_{\dg}^{\mathrm{co}}(\H^0N^{\eff}).
\]
Letting $\cM\cM_F$ and $\cM\cM_F^{\eff}$ be the categories of $\H^0N$- and $\H^0N^{\eff}$-comodules in finite-dimensional vector spaces, and $\cD_{\dg}(\cM\cM_F)$ and $\cD_{\dg}(\cM\cM_F^{\eff})$ the dg enhancements of their derived categories,
we would then have
\[
 \cM_{\dg,\bA^1}(F,\Q)\simeq \cD_{\dg}(\cM\cM_F),\quad \cM_{\dg,\bA^1}^{\eff}(F,\Q)\simeq \cD_{\dg}(\cM\cM_F^{\eff}),
\]
so existence of  a motivic $t$-structure would automatically realise $\cM_{\dg,\bA^1}(F,\Q)$ as the dg derived category of an abelian category of mixed motives, implying the $K(\pi,1)$ conjecture of \cite[(1.2)]{BlochKriz}.

Moreover, the full dg subcategory $\cM_{\dg}^{\eff,\heartsuit} \subset \cM_{\dg}^{\eff} $ consisting of objects $M$ with $\H^*E(M)$ concentrated in degree $0$ contains $\ker E$, so the heart of the $t$-structure would be
\[
 \cM_{\dg,\bA^1}^{\eff}(F,\Q)^{\heartsuit}\simeq  \cM_{\dg}^{\eff}(F,\Q)^{\heartsuit}/\ker E= \cM_{\dg}^{\eff}(F,\Q)^{\heartsuit}/(\ker \H^0E).
\]
Passing to the associated triangulated categories $\cM:= \H^0\cM_{\dg}$  would give 
\[
\cM\cM_F^{\eff}\simeq  \cM_{\bA^1}^{\eff}(F,\Q)^{\heartsuit} \simeq \cM^{\eff}(F,\Q)^{\heartsuit}/(\ker \H^0E).
\]

Nori's abelian category $\cM\cM_{\Nori}^{\eff}$ featuring in \cite[Remark \ref{HHtannaka1-motayoub1}]{HHtannaka1} is defined by forming a diagram $D^{\eff}$ of good pairs, and taking the universal $\Q$-linear abelian category under $D^{\eff}$ on which $\H^0E$ is faithful. The proof of \cite[Corollary 1.7]{HuberMuellerStach} gives a functor from $D^{\eff}$ to $ \H^0\cG^{\eff} \subset \cM^{\eff}_c(\Q,\Q)^{\heartsuit}$ and \cite[Proposition D.3]{HuberMuellerStach} gives a quasi-inverse to the induced functor
\[
 \cM\cM_{\Nori}^{\eff} \to \cM^{\eff}_c(\Q,\Q)^{\heartsuit}/(\ker \H^0E),
\]
so existence of a motivic $t$-structure would also imply
\[
 \cM\cM_{\Nori}^{\eff} \simeq \cM\cM_{\Q}^{\eff}. 
\]
Hanamura has shown in \cite{hanamuraMMotIII} that the existence of such a $t$-structure would follow from Grothendieck's standard conjectures, Murre's conjectures and a generalisation of the Beilinson--Soul\'e vanishing conjecture. 

Following \cite[Example \ref{HHtannaka1-mothomb}]{HHtannaka1}, all these coalgebras can be promoted to bialgebras, with the equivalences of categories preserving monoidal structures. In particular, $\H^0C$ is a naturally a Hopf algebra and $\H^0C^{\eff}$ a bialgebra. Writing $G_{\mot}(F):= \Spec \H^0C$ would then give a motivic Galois group, and letting $\cM\cM_F$ be the abelian tensor category of finite-dimensional $G_{\mot}(F)$-representations, the equivalence
\[
 \H^0\cM_{\dg,\bA^1}(F,\Q)\simeq \cD(\cM\cM_F), 
\] 
would become monoidal, as would the equivalence
\[
 \cM\cM_{\Q}\simeq \cM\cM_{\Nori}
\]
induced from $ \cM\cM_{\Q}^{\eff}\simeq \cM\cM_{\Nori}^{\eff}$ by stabilisation. In particular, this would imply that $G_{\mot}(\Q)$ is Nori's motivic Galois group from \cite[Theorem 1.14]{HuberMuellerStach}. This last result has since been proved in \cite{ChoudryGallauer} without assuming existence of a motivic $t$-structure.
\end{example}

\subsection{Tensor categories}\label{tensorsubsn} 

For the remainder of this section, we require that the field $k$ be of characteristic $0$.

\subsubsection{Non-negatively graded dg tensor categories}\label{nonnegtensn}

Assume that  $\cT$ is a rigid tensor dg category over $k$ (i.e. a symmetric monoidal dg category with strong duals), with $\cS:=\cT^0$ a rigid tensor subcategory. Define the dg $\cS^{\op}$-module $B \in \C(\cS^{\op})$ by $B(U):= \cT(U,{\mathbbm{1}})$. Note that tensor properties ensure that $\cT(U,V)= \cT(U\boxtimes V^{\vee},{\mathbbm{1}})= B(U\boxtimes V^{\vee})$. 

Thus $B \co \cS \to \C(k)$ is a symmetric lax monoidal functor, or equivalently a unital commutative algebra object in $\C(\cS^{\op})$. This is the same as saying that $B$ is a  DGA over $\cS$ in the sense of \cite[Definition 3.2]{higgs}.

If $\cS$ is a semisimple abelian category and $\omega \co \cS \to \FD\Vect$ a faithful symmetric monoidal functor, then \cite[Ch. II]{tannaka} shows that $\cS$ is equivalent to the category $\Rep(R)$ of finite-dimensional $R$-representations for the pro-reductive affine group scheme
\[
 R:= \Spec \omega^{\vee}\ten_{\cS} \omega.
\]
Equivalently, this is the category of finite-dimensional $O(R)$-comodules, where $O(R)$ is the Hopf algebra $\omega^{\vee}\ten_{\cS} \omega$.

As observed in \cite[Remark 3.15]{higgs}, the category of  DGAs over $\cS$ is equivalent to the category  of $R$-equivariant commutative dg algebras. Under this correspondence, $B$ corresponds to $A:= B(O(R))$ (regarding the right $R$-representation $O(R)$ as an object of $\ind(\cS)$, with the $R$-action on $A$ coming from the left action on $O(R)$). When this is concentrated in non-negative cochain degrees, note that it defines a schematic homotopy type in the sense of \cite{schematicv2}. For the inverse construction, we have $B(V)= A\ten^{R}V$  and $\cT(U,V)= \cS(U\ten A,V)= \Hom_R(V, U\ten_k A)$.

\begin{definition}
 Define $DG^{\ge 0}\Hopf_n\Alg_k $ to consist of (commutative but not necessarily cocommutative) dg Hopf algebras $C$ for which the underlying dg $k$-coalgebra lies in the category $DG^{\ge 0}\Co_n\Alg_k$ of Definition \ref{dgcoalgndef}. In other words, $C$ is concentrated in  non-negative cochain degrees, with $\H^0C \to C$ ind-conilpotent.
\end{definition}

\begin{proposition}\label{hopfalgconc1}
Take a $k$-linear  rigid tensor  dg category $\cT$ with  $\cT(X,Y)$ concentrated in non-negative degrees, $d\cT^0(X,Y)=0$ for all $X,Y$, and
$\cT^0$  a  semisimple rigid tensor subcategory. Assume that we have a strong monoidal $k$-linear functor $\omega \co \cT^0 \to \FD\Vect_k$. Then there is a model for the  bialgebra $C\simeq \omega^{\vee}\ten_{\cT}^{\oL}\omega$ of \S \ref{univbialg} with $C \in DG^{\ge 0}\Hopf_n\Alg_k$.
\end{proposition}
\begin{proof}
We just take the coalgebra $C$ from the proof of Proposition \ref{coalgconc1}, and observe that the formulae of \cite[\S \ref{HHtannaka1-monoidalenv}]{HHtannaka1} adapt to define a coproduct $\Delta$ and antipode $\rho$ on $C$, making it into a dg Hopf algebra.

Explicitly, writing $\cS= \cT^0$, the expression $\cT(U,V)= \cS(U\ten A,V)$ above allows us to rewrite 
\[
 D= \bigoplus_{n \ge 0}  A^{\ten n+2}\ten O(R)\quad C= \bigoplus_{n\ge 0} A^{\ten n}\ten O(R),
\]
with 
\[
 \Delta(a_1 \ten \ldots \ten a_n \ten 1)= \sum_{0 \le r \le n} (a_1 \ten \ldots \ten a_r \ten 1)\ten (a_{r+1} \ten \ldots \ten a_n \ten 1)
\]
and
\[
 \rho(a_1 \ten \ldots \ten a_n \ten 1)= (-1)^n(a_n \ten \ldots \ten a_1 \ten 1),
\]
and with multiplication given by the shuffle product. The coalgebra structure on $C$ is then given as the semidirect tensor product of $\bigoplus_{n\ge 0} A^{\ten n}$ and  $O(R)$.
\end{proof}

\begin{definition}\label{repRAdef}
 Given a commutative unital dg algebra $A$ in $R$-representations, define the dg category $\Rep(R,A)$ to have $R$-representations in finite-dimensional vector spaces as objects, with morphisms given by
\[
 \Rep(R,A)(U,V):=A\ten^R(U\ten_k V^{\vee}).
\]
Multiplication is induced by multiplication in $A$, with identities $1_A\ten \id_V \in A\ten^R(V\ten_k V^{\vee})$.
\end{definition}

\begin{remark}\label{cfhtpybar}
In the notation of \cite[Remark \ref{htpy-bargrmk}]{htpy}),  the dg Hopf algebra corresponding to $\cT$ is given by $O( R \ltimes \bar{G}(A))=O( R \ltimes \bar{G}(\cT(O(R),{\mathbbm{1}}))) $.

Then Propositions \ref{coalgconcequiv} and \ref{coalgconcper}  give  quasi-equivalences
\begin{eqnarray*}
 \cD_{\dg}^+(\Rep(R,A)) &\simeq& \cD_{\dg}^+(O( R \ltimes \bar{G}(A))),\\
\cD_{\dg}(\Rep(R,A)) &\simeq& \cD_{\dg}^{\mathrm{co}}(O( R \ltimes \bar{G}(A))).
 \end{eqnarray*}
 \end{remark}

\subsubsection{Koszul duality for tensor categories}

Now fix a pro-reductive affine group scheme $R$.

\begin{definition}
Let  $DG^{> 0}\Cat^{\ten}(R)$ be the category of rigid tensor dg categories $\cT$ in non-negative degrees with $\cT^0$ the category of finite-dimensional $R$-representations and $d\cT^0=0$.
\end{definition}
Note that under the discussion of \S \ref{nonnegtensn},   $DG^{> 0}\Cat^{\ten}(R)$ is equivalent to the category $DG^{> 0}\Comm(R)$ of commutative $R$-equivariant dg algebras $A$ in strictly positive cochain degrees. There is a model structure on $DG^{> 0}\Cat^{\ten}(R)$ 
in which weak equivalences are quasi-isomorphisms and fibrations are surjections.

\begin{definition}
 Dually, let $DG^{{\ge 0}}\Hopf_n\Alg(R)$ be the category of $R$-equivariant  ind-conilpotent dg Hopf algebras $N$ in non-negative cochain degrees. 
\end{definition}
Note that these correspond to Hopf algebras $C$ equipped with maps $O(R) \to C \to O(R)$, such that $O(R) \to C$ is ind-conilpotent. The correspondence sends $N$ to the tensor product $O(R) \ten N$, with comultiplication given by semidirect product. 

Moreover, as in \cite[Theorem B.4.5]{QRat}, ind-conilpotent dg Hopf algebras $N$ correspond to ind-conilpotent dg Lie coalgebras $L$, with $L= \tan(N)$. Thus $DG^{{\ge 0}}\Hopf_n\Alg(R)$ is equivalent to the category $DG^{{\ge 0}}\Co_n\Lie(R)$  of    ind-conilpotent dg Lie coalgebras $L$ in non-negative cochain degrees. It is thus equivalent to the category $dg\hat{\cN}(R)$ of \cite[Definition 4.1]{htpy}, so is a model for relative Malcev homotopy types over $R$.

By analogy with \cite[Theorem 3.1]{hinstack} and  \cite[Proposition 1.26]{ddt1}, there is then a model structure on $DG^{\ge 0}\Hopf_n\Alg_k $ in which weak equivalences are quasi-isomorphisms and cofibrations induce injections on tangent spaces in degrees $>0$. 

Now,  the functor $\beta$   of Definition \ref{betadef}  preserves tensor structures, so induces  $\beta \co DG^{> 0}\Cat^{\ten}(R) \to DG^{{\ge 0}}\Hopf_n\Alg(R)$. Equivalently, we have $\beta_{\ten} \co DG^{> 0}\Comm(R) \to DG^{{\ge 0}}\Co_n\Lie(R)$ given by setting $\beta_{\ten}(A)$ to be the cofree ind-conilpotent graded Lie coalgebra on cogenerators $A[1]$, with differential defined on cogenerators by
\[
 d_{\beta_{\ten}(A)}= (d_{A}, \circ) \co A[1] \oplus (\Symm^2A)[2] \to A[2],
\]
noting that $(\Symm^2A)[2] = \bigwedge^2(A[1])$. 

Moreover, the commutative and Lie operads are Koszul duals, and $\beta$ has 
 a left adjoint $\beta^*_{\ten}$, given by
\[
 \beta^*_{\ten}(L) = \bigoplus_{n>0} \Symm^n(L[-1]),
\]
with differential given on the generators by $d_L \oplus \Delta \co L[-1] \to L \oplus (\bigwedge^2L)[-1]$ where $\Delta$ is the Lie cobracket.

Thus Remark \ref{Koszulremark} and the comparisons above  allow us to adapt Proposition \ref{Koszulequiv} to give:
\begin{proposition}\label{Koszulequiv2}
 The functors $\beta^*_{\ten}\dashv \beta_{\ten}$ give a pair of Quillen equivalences between the categories $DG^{> 0}\Cat^{\ten}(R), DG^{{\ge 0}}\Hopf_n\Alg(R)$.
\end{proposition}

Using the characterisation of $DG^{> 0}\Cat^{\ten}(R)$ in terms of commutative dg algebras and  $DG^{{\ge 0}}\Hopf_n\Alg(R)$ in terms of dg Lie algebras, this result is effectively one of the equivalences of \cite[Theorem \ref{htpy-bigequiv}]{htpy}.

\subsubsection{Hearts of tensor $t$-structures}

\begin{proposition}\label{hopffilterprop}
 Take a $k$-linear rigid tensor dg category $\cT$ with the category $\H^0\cT$ abelian semisimple and $\H^{<0}\cT(X,Y)=0$ for all objects $X,Y$. Then $\cT$ is quasi-isomorphic to a rigid tensor dg category $\cB$ concentrated in non-negative degrees, with $d\cB^0(X,Y)=0$ for all $X,Y$.
\end{proposition}
\begin{proof}
 If we write $\cS:= \H^0\cT$, we see that the dg tensor categories $\cS$ and $\tau^{\le 0}\cT$ are quasi-isomorphic, from which it follows that $\cT$ is quasi-isomorphic to some rigid tensor  dg category $\cT'$ over $\cS$. Via the discussion of  \S \ref{nonnegtensn},  this is equivalent to giving a  commutative  dg algebra $A$ in $\C(\cS^{\op})$ with $\H^0A=k$.

We now apply the Rees algebra construction $\zeta(A, \tau)$ to the good truncation filtration on $A$, regarding the Rees algebra as a deformation of $\zeta(A, \tau)/t= \H^*(A)$. The bar-cobar resolution $\beta^*_{\ten}\beta_{\ten}(\H^{>0}A)$ of Proposition \ref{Koszulequiv2} gives a cofibrant replacement for $\H^{>0}A$ concentrated in strictly positive degrees. The proof of Proposition \ref{filterprop} now adapts, substituting Andr\'e--Quillen cohomology for Hochschild cohomology.
\end{proof}

\begin{corollary}\label{hopfalgconc3}
Take a $k$-linear rigid tensor dg category $\cT$ with the category $\H^0\cT$ abelian semisimple and $\H^{<0}\cT(X,Y)=0$ for all objects $X,Y$. Assume that we have a lax monoidal $k$-linear dg  functor $\omega \co \cT \to \C_{\dg}(k)$ with $\H^i\omega(X)=0$ for all $i \ne 0$, $\H^0\omega(X)$ finite-dimensional for all $X \in \cT$, and quasi-strong in the sense that the structure maps
\[
 \omega(X)\ten_k \omega(Y) \to \omega(X\ten Y), \quad k \to \omega({\mathbbm{1}}) 
\]
quasi-isomorphisms for all $X,Y \in \cT$.

Then there is a dg Hopf algebra $C\in DG^{\ge 0}\Hopf_n\Alg_k$ with $C \simeq \omega^{\vee}\ten_{\cT}^{\oL}\omega$,   together with a tensor dg functor $ \cD_{\dg}(\cT) \to \C_{\dg}(C)$ inducing
   quasi-equivalences  $\cD_{\dg}(\cT)/\ker \omega \simeq  (\ker \omega)^{\perp} \simeq \cD_{\dg}(C)$.  Here, $\C_{\dg}(C)$ and $\cD_{\dg}(C)$ are defined using the coalgebra (not the algebra) structure of $C$. 

If the functor $\H^0\omega \co \H^0\cT \to \FD\Vect_k$  is faithful, then these induce  quasi-equivalences $\cD_{\dg}^+(\cT) \simeq \cD_{\dg}^+(C)$ and $\cD_{\dg}(\cT) \simeq \cD_{\dg}^{\mathrm{co}}(C)$.
\end{corollary}
\begin{proof} 
 By Proposition \ref{hopffilterprop}, there is a tensor quasi-isomorphism $q \co\cB \to \cA$ with $\cB$ satisfying the conditions of Proposition \ref{coalgconc1}. Write $\cS:= \cB^0 \simeq \H^0\cA$. 

We now show how to replace $\omega \circ q\co \cB \to \C_{\dg}(k)$ with $\H^0(\omega \circ q)$. If  $A$ is the  commutative  dg algebra in $\C(\cS^{\op})$ corresponding to $\cB$, then $\omega \circ q$ corresponds to an $A$-algebra $M$ in $\C(\cS^{\op}) $. Since $\H^*M(X)$ is concentrated in degree $0$ for all $X$, we have quasi-isomorphisms $M \la \tau^{\le 0}M \to \H^0M$ of algebras in $\C(\cS^{\op}) $. As $A$ is cofibrant, the map $A \to M$ is homotopic to a map taking values in $\tau^{\le 0}M$, so $M$ and $\H^0M$ are quasi-isomorphic as $A$-algebras.

Since we now have a monoidal functor $\H^0(\omega \circ q)\co \cB^0 \to \FD\Vect_k$, the
 first statement of the corollary follows from  Corollary \ref{coalgconc2},  with \cite[Proposition \ref{HHtannaka1-monoidalprop}]{HHtannaka1} ensuring that the tensor structure is preserved. The second statement is then an immediate consequence of Propositions \ref{coalgconcequiv} and \ref{coalgconcper}.
\end{proof}

\subsection{Comparison with Moriya}\label{moriyasn}

\begin{definition}\label{hatH0def}
 For a dg category $\cA$ satisfying the conditions of Corollary \ref{coalgconc2}, write $\widehat{\H^0\cA} \subset \cD(\cA)$ for the subcategory generated by $\H^0\cA$ under finite extensions (but not by suspensions). This is the completion functor of \cite[\S 2.2]{moriya}. 
\end{definition}
Note that all objects of $\widehat{\H^0\cA}$ are perfect, so we also have a natural embedding $\ind(\widehat{\H^0\cA}) \to \cD(\cA)$. Under the conditions of Proposition \ref{coalgconcequiv}, recall that  Proposition \ref{coalgconcper} gives an equivalence between $\cD^+(\cA)$ and $\cD^+(C)$, with $C$ concentrated in non-negative degrees, 
and that $\H^0\cA$ is equivalent to the category of  semisimple $\H^0C$-comodules in finite-dimensional vector spaces. It then follows that $\widehat{\H^0\cA}$ is equivalent to the category of all   $\H^0C$-comodules in finite-dimensional vector spaces, and $\ind(\widehat{\H^0\cA})$ is equivalent to the category of all   $\H^0C$-comodules in  vector spaces.

 In \cite[Definition 3.1.1]{moriya}, Moriya  gives a notion of Tannakian dg category. Given a rigid dg tensor category $\cA$ satisfying the conditions of Corollary \ref{coalgconc2} with a tensor functor $\omega \co \cA^0 \to \FD\Vect$ such that $\omega|_{\H^0\cA}$ is faithful (as in Proposition \ref{coalgconcequiv}),
  we could form a Tannakian dg category in Moriya's sense by taking $\hat{\cA}$ to be the full dg subcategory of $\cD_{\dg}(\cA)$ on objects $\widehat{\H^0\cA}$.

Note that because $\hat{\cA}$ is a full generating subcategory of $\per_{\dg}(\cA)$, it follows that $\per_{\dg}(\cA) \to \per_{\dg}(\hat{\cA})$ is a quasi-equivalence, so Theorem \ref{tannakathm} gives the same output for both $\cA$ and $\hat{\cA}$. In topological contexts, this just amounts to saying that semisimple local systems generate all local systems under extension.

In \cite[\S 3.3]{moriya}, functors $T^{\ss}$ and $T$ are defined  from  commutative unital dg algebras $A$ in $R$-representations to  dg categories. In fact, $T^{\ss}(R,A)= \Rep(R,A)$ as given in Definition \ref{repRAdef}, and $T(R,A)= \widehat{T^{\ss}(R,A)}$. We have therefore defined Morita equivalences
\[
 \Rep(R,A) \to T(R,A) \to\per_{\dg}(\Rep(R,A)),
\]
so 
$\Rep$ and $T$ give rise to the same theory.
However, $T(R,A)$ feels like a halfway house between the minimal choice $\Rep(R,A)$ and the fibrant replacement $\per_{\dg}(\Rep(R,A)) $.

Moriya's analogue of the construction in Theorem \ref{tannakathm} is the construction $A_{\red}$ of \cite[Definition 3.3.3]{moriya}, but this has only limited  functoriality, which is a well-known limitation of working with equivariant DGAs. By allowing dg coalgebras to have negative terms,  \cite[Example \ref{HHtannaka1-HHunivbialg}]{HHtannaka1} gives us a completely functorial choice of the dg Hopf algebra $C(\cT, \omega)$ corresponding under \cite[Theorem \ref{htpy-bigequiv}]{htpy} to Moriya's $A_{\red}(\cT, \omega)$. 

\section{Schematic and Relative Malcev homotopy types}\label{htpysn}


\subsection{de Rham  homotopy types}\label{derhamsn}

 Take a pointed connected manifold $(X,x)$, and choose a full rigid tensor subcategory $\cS$ of the category of real finite-dimensional semisimple local systems on $X$. Let $\cT$ be the real dg tensor category with the same objects as $\cS$, but with morphisms
\[
 \cT(U,V)= A^{\bt}(X, U\ten V^{\vee}),
\]
where $A^{\bt}(X,-)$ is the de Rham complex. Note that $\H^0\cT \simeq \cS$.

Now, the basepoint $x$ defines a fibre functor $x^* \co \cT^0 \to \FD\Vect_{\R}$ sending $U$ to $U_x$. We are therefore in the setting of Corollary \ref{hopfalgconc3}. Moreover, $x^*\co \cS \to \FD\Vect$ is faithful because  $X$ is connected, so the conditions of Proposition \ref{coalgconcequiv} are satisfied.

Let  $R$ be the real pro-algebraic group $\Spec ((x^*)^{\vee}\ten_{\cS}x^*)$; as in \S \ref{tensorsubsn}, we have an equivalence $x^*\co \cS \to \Rep(R)$ of tensor categories. Equivalently, we have a Zariski-dense group homomorphism $\rho \co \pi_1(X,x) \to R(\R)$. As in \cite{htpy}, write $\bO(R)$ for the local system corresponding to the right $R$ ind-representation $O(R)$. The $R$-equivariant dg algebra $A$ from \S \ref{tensorsubsn} is then just the dg algebra
\[
 A^{\bt}(X, \bO(R))
\]
  of equivariant cochains from \cite[Definition \ref{htpy-bo}]{htpy}, so $\cT$ is equivalent to $\Rep(R, A^{\bt}(X, \bO(R)))$.  

\begin{definition}
 Write $\sA_X^{\bt}$ for the sheaf of real $\C^{\infty}$ differential forms on $X$, regarded as a sheaf of dg algebras with standard differential $\sA^n_X \to \sA^{n+1}_X$. Note that $A^{\bt}(X,U)=\Gamma(X, U\ten_{\R}\sA^{\bt}_X)$.
\end{definition}

\begin{definition}\label{derivedconns}
 Define the  dg category $\cP(X)$ to consist of locally perfect $\sA^{\bt}_X$-modules in complexes of sheaves on $X$. 
 Define $\cP(X, \cS)$ to be the full dg subcategory of $\cP(X)$ generated under shifts and extensions by objects of the form $U\ten_{\R}\sA^{\bt}_X$ for $U \in \cS$.
 \end{definition}
 
Note that because $\sA^{\bt}_X$ is a flabby resolution of $\R$, the dg category $\cP(X)$ is quasi-equivalent to the category of locally constant hypersheaves in real complexes on $X$.

\begin{lemma}\label{PXlemma}
 When $\cS$ consists of all semisimple local systems, we have $\cP(X,\cS)= \cP(X)$.
\end{lemma}
\begin{proof}
Given $\sV^{\bt} \in \cP(X)$, the sheaf $\bar{\sV}^{\bt} :=\sV^{\bt}\ten_{\sA^{\bt}_X}\sA^0_X$ is a finite rank complex of $\C^{\infty}$-vector bundles on $X$. We then form the good truncation filtration $\{\tau^{\le n}\bar{\sV}^{\bt}\}_n$. Now, the morphisms $\sA^0_X \to \sA^*_X \to \sA^0_X$ of sheaves of dg algebras give an isomorphism
\[
 \sV^*\cong \bar{\sV}^*\ten_{\sA^0_X}\sA^*_X       
\]
of graded $\sA^*_X$-modules (where we write $U^*$ for the graded object underlying a complex $U^{\bt}$). We then define an increasing  filtration $\{W_n\sV^*\}_n$ on $\sV^*$ by
\[
     W_n\sV^*=    (\tau^{\le n}\bar{\sV}^*)\ten_{\sA^0_X}\sA^*_X. 
\]

Writing $\nabla$ for the differential on $\sV^{\bt}$ and $\delta$ for the differential on $\bar{\sV}^{\bt}$, we can set $\nabla= \delta+D$, for some $D\co \sV^n \to \bigoplus_{i\ge 0} \sV^{n-i}\ten_{\sA^0_X} \sA^{i+1}_X$. By construction, $\delta W_n \subset W_n$, and we automatically have $D W_n \subset W_n$, so $\{W_n\sV^{\bt}\}_n$ defines a filtration on $\sV^{\bt}$.

Let $\sU^{\bt}$ be the quotient $W_n\sV^{\bt}/W_{n-1}\sV^{\bt}$;
 the  only non-zero terms of $\bar{\sU}$ are $\bar{\sU}^{n-1}, \bar{\sU}^n$. For $\nabla=D + \delta$ as above, we have  flat connections $D \co \bar{\sU}^i \to \bar{\sU}^i\ten_{\sA^0_X}\sA^1$  and a map $\delta \co \sU^{n-1} \to \sU^n$ commuting with $D$. 
 Thus there exist local systems $U^{n-1}, U^n$ with   $(\bar{\sU}^i, D)= (U^i\ten_{\R}\sA^0_X, \id_U\ten d)$. Because any local system is an extension of semisimple local systems, it follows that the $\sF_i:= (\bar{\sU}^i\ten_{\sA^0_X}\sA^{\bt}_X , \id \ten d \pm D\ten \id)$ lie in $\cP(X,\cS)$. Since $\sU^{\bt}$ is an extension of $\sF_{n-1}[1-n]$ by $\sF_n[-n]$, we also have $ W_n\sV^{\bt}/W_{n-1}\sV^{\bt}\in \cP(X,\cS)$. As the filtration is finite ($\bar{\sV}^*$ being of finite rank), this means that $\sV^{\bt} \in \cP(X,\cS)$, as required.
\end{proof}

\begin{lemma}\label{integratelemma}
The dg categories $\per_{\dg}(\cT)$ and $\cP(X,\cS)$ are quasi-equivalent. 
\end{lemma}
\begin{proof}
 We define a dg functor $\iota\co \cT \to \cP(X,\cS)$ by sending $U$ to $U\ten_{\R}\sA^{\bt}_X$. The dg category $\cP(X,\cS)$  is closed under shifts, extensions and direct summands, so $\cP(X,\cS)$ is Morita fibrant --- in other words,  $\cP(X,\cS) \to \per_{\dg}(\cP(X,\cS))$ is a quasi-equivalence.

Now, the dg functor $\iota$ is clearly full and faithful, since the maps $\cT(U,V) \to \cP(X)(\iota U, \iota V)$ are isomorphisms. The definition of $\cP(X,\cS)$ ensures that it is generated by $\iota \cT$, so $\per_{\dg}(\cT) \to \per_{\dg}(\cP)$ must be a quasi-equivalence. 
\end{proof}

Note that Moriya's category $\hat{\cT}$ (see \S \ref{moriyasn}) embeds in $\cP(X)$ as the full dg subcategory on objects $\sV^{\bt}$ with $\bar{\sV}^{\bt}$ concentrated in degree $0$ --- these correspond to flabby resolutions of local systems.

\subsection{Betti homotopy types}

We now let $k$ be a field of characteristic $0$.

\begin{definition}
As in \cite[Definition \ref{htpy-zardense}]{htpy}, define the relative Malcev homotopy type $G(X,x)^{R,\mal}$ of a pointed connected topological space $(X,x)$ with respect to a Zariski-dense representation $\rho \co \pi_1(X,x) \to R(k)$
as follows. First form the reduced simplicial set $\Sing(X,x)$ of singular simplices based at $x$, then apply Kan's loop group functor from \cite{loopgp} to give a simplicial group $G(X,x):= G(\Sing(X,x))$. Note that $\pi_0G(X,x)= \pi_1(X,x)$, and apply the relative Malcev completion construction of \cite{hainrelative} levelwise to $G(X,x) \to R(k)$, obtaining a simplicial affine group scheme
\[
 G(X,x)^{R,\mal},
\]
with each $G(X,x)_n^{R, \mal}$ a pro-unipotent extension of $R$.

In other words, $G(X,x)_n \to (G(X,x)^{R, \mal})_n(k) \xra{f(k)} R(k)$ is the universal diagram with $f$ a pro-unipotent extension.
\end{definition}

To a relative Malcev homotopy type $G(X,x)^{R,\mal}$ are associated relative Malcev homotopy groups $\varpi_n(X^{R,\mal},x):= \pi_{n-1}G(X,x)^{R,\mal}$. These are affine group schemes, with $\varpi_1(X^{R,\mal},x) = \pi_1(X,x)^{R, \mal}$. The higher homotopy groups are pro-finite dimensional vector spaces, and are often just $\pi_n(X,x)\ten_{\Z}k$ --- see \cite[Theorem \ref{htpy-classicalpimal}]{htpy}, \cite[Theorem \ref{mhs2-fibrations}]{mhs2} and \cite[Theorem \ref{weiln-etpimal}]{weiln}.

\begin{examples}\label{nilpmalcev}
 When $\cS$ is the category of all semisimple local systems in $k$-vector spaces on $X$, we write $G(X,x)^{R, \mal}= G(X,x)^{\alg}$.  Note that \cite[Corollary \ref{htpy-eqtoen}]{htpy} shows that $G(X,x)^{\alg}$ is a model for To\"en's schematic homotopy types.

When $\cS$ is the category of constant local systems on $X$, note  that $R=1$ and that $G(X,x)^{1, \mal}$ is  the nilpotent $k$-homotopy type, so Quillen's rational  homotopy type from \cite{QRat} when $k=\Q$.
\end{examples}

We now specialise to the setting of the previous section, with $k=\R$.
\begin{proposition}\label{propforms}
 When $X$ is a manifold   and $\cT= \Rep(R,A^{\bt}(X, \bO(R))) $, the dg Hopf algebra $C\simeq (x^*)^{\vee}\ten^{\oL}_{\cT}x^*$ of Corollary \ref{hopfalgconc3} associated to the fibre functor $x^*\co \cT \to \FD\Vect$ is a model for the relative Malcev homotopy type $G(X,x)^{R,\mal}$ of $(X,x)$ under the equivalences of \cite[Theorem \ref{htpy-bigequiv}]{htpy}.
\end{proposition}
\begin{proof}
We need to show that  the Dold--Kan denormalisation functor $D$ (\cite[Definition \ref{htpy-DKD}]{htpy}) from dg Hopf algebras to cosimplicial Hopf algebras sends $C$ to a model for the ring of functions on the simplicial group scheme $G(X,x)^{R,\mal}$. By Remark \ref{cfhtpybar},  the dg Hopf algebra $C$ is given by $O( R \ltimes \bar{G}(A))$, for $A= A^{\bt}(X, \bO(R))$. Applying $D$ then gives 
\[
 DC= O( R \ltimes \bar{G}(DA)),
\]
where $\bar{G}$ is now the functor on cosimplicial algebras defined in \cite[Definition \ref{htpy-qdef}]{htpy}. [Equivalently, this is a weak equivalence
\[
 B\Spec DC \simeq [(\Spec DA)/R]
\]
of affine stacks in the sense of \cite{chaff}, where $B$ is the nerve.]

By \cite[Proposition \ref{htpy-propforms}]{htpy}, the simplicial group scheme $G(X,x)^{R,\mal}$ is quasi-isomorphic to $R \ltimes \bar{G}(DA)$, so we have shown
\[
\Spec  DC \simeq G(X,x)^{R,\mal}.
\]
\end{proof}

\begin{remark}\label{singDGcat}
 For any reduced simplicial set $X$ and Zariski-dense representation $\rho\co \pi_1(X) \to R(k)$, there is a relative Malcev homotopy type $G(X)^{R,\mal}$. By \cite[Theorem \ref{htpy-eqhtpy}]{htpy}, this homotopy type corresponds (via \cite[Theorem \ref{htpy-bigequiv}]{htpy}) to the $R$-equivariant cosimplicial algebra
\[
 \CC^{\bt}(X, \rho^{-1}O(R))
\]
of equivariant singular cochains with coefficients in the local coefficient system $\rho^{-1}O(R) $ (with right multiplication). 

A model for the corresponding $R$-equivariant dg algebra is given by applying the Thom--Sullivan functor $\Th$. The corresponding dg tensor category $\cT$ has finite-dimensional $R$-representations as objects, and morphisms
\[
 \cT(U,V)= \Th \CC^{\bt}(X, \rho^{-1}(U\ten V^{\vee})).
\]
When $R= \pi_1(X)^{\red}$ is the reductive pro-algebraic fundamental group of $X$, the quasi-isomorphism between cosimplicial and cocubical cochains gives a quasi-isomorphism between $\hat{\cT}$  and the dg category $T_{\dR}(X)$ of \cite[Theorem 1.0.4]{moriya}. 
\end{remark}

\begin{corollary}\label{integratecor}
 The dg category $\cD_{\dg}^{\mathrm{co}}(O( G(X,x)^{R,\mal}))$ of Definition \ref{Dcodef} is quasi-equivalent to $\ind(\cP(X,\Rep(R)))$, for $\cP(X, \Rep(R))$  the dg category of derived connections from  Definition \ref{derivedconns}. 
Under this equivalence, $\cP(X, \Rep(R))$ corresponds to the full dg subcategory of $\cD_{\dg}(O( G(X,x)^{R,\mal}))$ on  fibrant replacements of finite-dimensional comodules.
 The equivalence  respects the tensor structures.
\end{corollary}
\begin{proof}
Remark \ref{cfhtpybar} gives a  quasi-equivalence
\[
 \cD_{\dg}(\Rep(R,A)) \simeq \cD_{\dg}^{\mathrm{co}}(O( R \ltimes \bar{G}(A))),
\]
which is compatible with tensor structures by Corollary \ref{hopfalgconc3}.
Proposition \ref{propforms} gives $\cD_{\dg}^{\mathrm{co}}(O( R \ltimes \bar{G}(A)))\simeq \cD_{\dg}^{\mathrm{co}}(O( G(X,x)^{R,\mal}))$, while Lemma \ref{integratelemma} gives $\per_{\dg}(\Rep(R,A))\simeq \cP(X,\Rep(R))$ and hence $ \cD_{\dg}(\Rep(R,A)) \simeq \ind(\cP(X,\Rep(R)))$. Combining these gives the tensor quasi-equivalence $\cD_{\dg}^{\mathrm{co}}(O( G(X,x)^{R,\mal})) \simeq \ind(\cP(X,\Rep(R))) $.

For the characterisation of $\cP(X,\Rep(R))$, we appeal to Proposition \ref{coalgconcper}.
\end{proof}

\begin{remark}\label{generalbaseBetti}
 As in \cite[Remark \ref{HHtannaka1-generalbase}]{HHtannaka1}, we can also consider  a finite set $T$ of basepoints. Proposition \ref{propforms} then adapts to show that the dg Hopf algebroid $C_T$ given by  $C_T(x,y)\simeq (x^*)^{\vee}\ten^{\oL}_{\cT}y^*$  is a for the unpointed relative Malcev homotopy type $G(X;T)^{R,\mal}$ of $X$, where $G(X;T)$ is the restriction of Dwyer and Kan's loop groupoid $G(X)$ (from \cite{pathgpd}) to the set $T$ of objects. 

Because these dg Hopf algebroids are all equivalent as $T$ varies (or equivalently, because the fibre functors are all quasi-isomorphic), taking the colimit over all finite $T$ gives a model for $G(X)^{R,\mal}$.

Points of $X$  also give a set $\{\bar{x}^*\}$ of fibre functors  on the category of all $k$-linear sheaves, not just on locally constant sheaves. Any such set $T$ of points  yields a dg bialgebroid $C'_T$, but the dg derived category of $C'_T$-comodules is then just monoidally quasi-equivalent to the dg derived category of $k$-linear sheaves supported on $T$. Because the  site has enough points, the set of all points gives a jointly faithful set of fibre functors on the category of $k$-linear sheaves. However, \cite[Remark \ref{HHtannaka1-generalbase}]{HHtannaka1} only applies to finite sets of fibre functors, so  only  finitely supported $k$-linear sheaves arise as comodules of the associated dg bialgebroid $C'= \LLim_T C'_T$. 
\end{remark}

\subsection{The universal Hopf algebra}\label{univhalg}

An unfortunate feature of relative Malcev homotopy types is that they rely on a choice of basepoint(s). However, the constructions of
\S \ref{univbialg} 
give us a universal bialgebra construction $D(X,\cS)$ associated to a topological space $X$ and a tensor category $\cS$ of semisimple local systems. This should be regarded as the ring of functions on the  space of  algebraic paths generated by $\cS$, while $G(X,x)^{R, \mal}$ is the loop group at a fixed basepoint.

In order to understand $D(X,\cS)$, we must first understand the category $\cD_{\dg}(\cT\ten \cT^{\op}) $ in which it lives. When $X$ is a manifold, recall that $\cT= \Rep(R,A^{\bt}(X, \bO(R)))$ and  $\Rep(R)\simeq\cS= \H^0\cT$. By Lemma \ref{integratelemma}, we have a  quasi-equivalence
\begin{align*}
 \per_{\dg}(\cT\ten \cT^{\op}) &\simeq \cP(X^2, \Rep(R^2)),\\
(U, V) &\mapsto \iota_{X^2}((\pr_1^{-1}U)\ten_k (\pr_2^{-1}V^{\vee})),
\end{align*}
and hence
\[
 \cD_{\dg}(\cT\ten \cT^{\op}) \simeq \ind(\cP(X^2, \Rep(R^2))).
\]

Understanding $\C_{\dg}(\cT\ten \cT^{\op}) $ is harder, but observe that there is a dg functor $r$ from the dg category of  $\sA^{\bt}_{X^2}$-modules to $\C_{\dg}(\cT\ten \cT^{\op}) $, given by
\[
 (rM)(U,V)= \HHom_{\sA^{\bt}_{X^2}}(\iota_{X^2}((\pr_1^{-1}U^{\vee})\ten_k (\pr_2^{-1}V)) ,M).
\]

Now, note that for $K,L \in \cP(X)$, we have
\[
 \HHom_{\cP(X) }(K,L)\cong \HHom_{\sA^{\bt}_{X^2}}((\pr_1^*K)\ten_{\sA^{\bt}_{X^2}} (\pr_2^*L^{\vee}),  \Delta_*\iota_X(k)),
\]
where $\Delta \co X \to X \by X$ is the diagonal morphism. Thus
\[
 \id_{\cA}=r \Delta_*\iota_X(k) \in \C_{\dg}(\cT\ten \cT^{\op}).
\]
Likewise,
\begin{eqnarray*}
&& \HHom_{\sA^{\bt}_{X^2}}(\iota_{X^2}((\pr_1^{-1}U)\ten_k (\pr_2^{-1}V^{\vee})), M\ten_{\cT}N)\\
&=& M(U,-)\ten_{\cT}N(-,V)\\
&\cong&\HHom_{\sA^{\bt}_{X^3}}(\iota_{X^2}((\pr_1^{-1}U)\ten_k (\pr_3^{-1}V^{\vee})), (\pr_{12}^*M)\ten_{\sA^{\bt}_{X^3}}(\pr_{23}^*N  ))\\
&\cong &  \HHom_{\sA^{\bt}_{X^2}}(\iota_{X^2}((\pr_1^{-1}U)\ten_k (\pr_2^{-1}V^{\vee})), \pr_{12*}((\pr_{12}^*M)\ten_{\sA^{\bt}_{X\by X\by X}}(\pr_{23}^*N  ) ),
\end{eqnarray*}
so we have
\[
 M\ten_{\cT}N = r \pr_{13*}((\pr_{12}^*M)\ten_{\sA^{\bt}_{X^3}}(\pr_{23}^*N  )     ).
\]

Combining these results gives:
\begin{lemma}\label{pathspacelemma}
 A universal bialgebra $D(X,\cS)$ corresponds under the equivalence $\cD_{\dg}(\cT\ten \cT^{\op}) \simeq \ind(\cP(X^2, \Rep(R^2)))$ above to a sheaf $\bD \in \ind(\cP(X^2, \Rep(R^2)))$ equipped with a commutative unital multiplication 
\[
 \bD\ten_k\bD \to \bD
\]
and a coassociative $\sA^{\bt}_{X^2}$-linear  comultiplication
\[
 \bD \pr_{13*}((\pr_{12}^*\bD)\ten_{\sA^{\bt}_{X^3}}(\pr_{23}^*\bD  )     )
\]
with $\sA^{\bt}_{X^2}$-linear counit 
\[
 \bD\to \Delta_*\iota_X(k).
\]
\end{lemma}
Beware that although the co-unit $\bD \to r\Delta_*\iota_X(k)$ is  a quasi-isomorphism of sheaves, the induced map $ \bD \to \Delta_*\iota_X(k)$ is far from being so, with the object on the left locally constant  and that on the right supported on the diagonal. In some sense, $\bD$ is the universal  coalgebra under $\Delta_*k$ generated by $\cS\ten\cS^{\op}$. In the same way that a path space in topology is a fibrant replacement for the diagonal, $\bD$ is a cofibrant replacement for functions on the diagonal, which is why we think of it as functions on the  space of  algebraic paths generated by $\cS$.

\begin{example}\label{concHHbialgex}
Note that the construction of Propositions \ref{coalgconc1} and \ref{hopfalgconc1} gives an efficient choice
$
\iota_{X^2} N\CCC(\cT/\cS, i^{\op}\ten i) 
$
for the dg bialgebra $\bD$, in which case it becomes a dg Hopf algebra. Explicitly, we have
\[
 \iota_{X^2} N\CCC(\cT/\cS, i^{\op}\ten i)=N\CCC(\cT/\cS, (\pr_1^{-1}\iota_Xi^{\op})\ten (\pr_2^{-1}\iota_Xi)) \ten_{[(\pr_1^{-1}\sA^{\bt}_X)\ten (\pr_2^{-1}\sA^{\bt}_X)]} \sA^{\bt}_{X^2}, 
\]
where
\begin{align*}
& \uline{\CCC}_n(\cT/\cS, (\pr_1^{-1}\iota_Xi^{\op})\ten (\pr_2^{-1}\iota_Xi))=\\
 &(\pr_1^{-1}\iota_X\bO(R))\ten^R \underbrace{A^{\bt}(X, \bO(R))\ten^R \ldots \ten^R A^{\bt}(X, \bO(R))}_n\ten^R(\pr_2^{-1}\iota_X\bO(R)).
\end{align*}
 Thus $\uline{\CCC}_0(\cT/\cS, i^{\op}\ten i)=\iota_{X^2}(\bO(R)\ten^R\bO(R))$, which is quasi-isomorphic to the local system given by the $\pi_1(X)^2$-representation $O(R)$, with the two copies of $\pi_1(X)$ acting as left and right multiplication.
\end{example}

\begin{example}\label{irredHHrmk2}
Following Remark \ref{irredHHrmk}, for $i \co \cS \to \cT$ we may describe $N\CCC(\cT/\cS, i^{\op}\ten i)$ in terms of irreducibles.  Let $\{V_{\alpha}\}_{\alpha}$ be a set of irreducible objects of complex $R$-representations, with one in each isomorphism class. Complex conjugation $\Gal(\Cx/\R)$ acts on this set, and  then we have
\begin{align*}
& \uline{\CCC}_n(\cT/\cS, (\pr_1^{-1}\iota_Xi^{\op})\ten (\pr_2^{-1}\iota_Xi))\ten_{\R}\Cx\cong \\
&\bigoplus_{\alpha_0, \ldots, \alpha_n}  (\pr_1^{-1}\sA_X^{\bt}(V_{\alpha_0}^{\vee}))\ten_{\Cx}A^{\bt}(X,V_{\alpha_0}\ten_{\Cx}V_{\alpha_1}^{\vee})\ten_{\Cx}\ldots \ten_{\Cx}A^{\bt}(X,V_{\alpha_{n-1}}\ten_{\Cx}V_{\alpha_n}^{\vee})\ten_{\Cx}(\pr_2^{-1}\sA_X^{\bt}(V_{\alpha_n})),
\end{align*}
with $\uline{\CCC}_n(\cT/\cS, (\pr_1^{-1}\iota_Xi^{\op})\ten (\pr_2^{-1}\iota_Xi))$ given by taking $\Gal(\Cx/\R)$-invariants.

When $\cS$ is the category of constant local systems, corresponding to the real (nilpotent) homotopy type as in Examples \ref{nilpmalcev},  this simplifies to
\begin{align*}
 & \iota_{X^2} \uline{\CCC}_n(\cT/\cS,i^{\op}\ten i )\\
&\cong[\pr_1^{-1} \sA_X^{\bt}\ten_{\R}(A^{\bt}(X,\R))^{\ten n}\ten_{\R}\pr_2^{-1}\sA_X^{\bt}]\ten_{[(\pr_1^{-1}\sA^{\bt}_X)\ten (\pr_2^{-1}\sA^{\bt}_X)]} \sA^{\bt}_{X^2}\\
&=(A^{\bt}(X,\R))^{\ten n}\ten_{\R}\sA^{\bt}_{X^2}.
\end{align*}
In other words,
\[
 \iota_{X^2} \uline{\CCC}_{\bt}(\cT/\cS,i^{\op}\ten i )= \uline{\CCC}_{\bt}(A^{\bt}(X,\R), \sA^{\bt}_{X^2}).
\]
\end{example}

\begin{remark}\label{symmetry1rk}
Consider the case of a group $G$ acting on a manifold  $X$, with $\cS$ a $G$-equivariant rigid tensor subcategory of semisimple local systems (so $g^*U \in \cS$ whenever $U \in \cS$ and $g \in G$). Then we have an action of $G$ on $\bD$ over $X \by X$, with respect to the diagonal action of $G$ on $X \by X$. This is because $G$-equivariance of $\cS$ gives an action of $G$ on $\bO(R)$ over $X$ (i.e. compatible isomorphisms $\bO(R) \cong g^*\bO(R)$ for all $g \in G$), and hence an action on $A^{\bt}(X, \bO(R))$. For well-behaved $G$-actions, this allows us to regard $\bD$ as a sheaf of dg algebras on the quotient $(X \by X)/G$.
When $x \in X$ is a fixed point for the $G$-action, note that the dg Hopf algebra $C=(x,x)^*\bD$ inherits a $G$-action from $\bD$.  

Of course, in order to define a $G$-action on $\bD$, it suffices to have $G$-actions on $\cS$ and on the relative Malcev homotopy type $A^{\bt}(X, \bO(R))$. 
When $X$ is a compact K\"ahler manifold,  \cite[Theorem \ref{mhs2-mhspin}]{mhs2} show that  the Tannakian fundamental group  $\Pi(\MTS)$ of the category of mixed twistor structures acts algebraically on $A^{\bt}(X, \bO(R))$ for all $R$, with trivial action on $\cS$ (and hence  $\bO(R)$). This gives an algebraic action of $\Pi(\MTS)$ on $\bD$, so would allow us to regard  $\bD$ as an object of the derived category of mixed twistor modules, compatibly with the Hopf algebra structure. When the local systems in $\cS$ all underlie variations of Hodge structure,  there is also an algebraic circle action on $\cS$ and on $A^{\bt}(X, \bO(R))$, combining with the $\Pi(\MTS)$-action to give an action of the Tannakian fundamental group  $\Pi(\MHS)$ of the category of mixed Hodge  structures. Then $\bD$ would lie in  the derived category of mixed Hodge modules.
\end{remark}

\begin{remark}\label{pinlocsys}
Note that  the cohomology sheaf $\sH^0\bD$ is the local system $\bO(\varpi_1X^{R, \mal})$ on $X \by X$ defined in \cite[Corollary \ref{mhs2-kvmhspin}]{mhs2} as corresponding to  $O(\varpi_1(X^{R, \mal}, x))$ with its left and right actions by   $\pi_1(X, x)$. All of the cohomology sheaves of $\bD$ are necessarily local systems. 

The pullback $\Delta^*\bD$ to the diagonal is a sheaf of Hopf algebras   (the ring of functions on the  space of  algebraic loops generated by $\cS$). Then 
the higher cohomology sheaves of the sheaf of primitive elements of $\Delta^*\bD $ are dual to the local systems $\Pi^n(X^{\rho, \mal})$ of \cite[Corollary \ref{mhs2-kvmhspin}]{mhs2} corresponding to the relative Malcev homotopy groups
$\varpi_n(X^{R, \mal}, x)$ with their adjoint actions by $\pi_1(X, x)$.

When $f\co X \to Y$ is a fibration with section $p$, choose $R$ so  that $\Rep(R)$ contains the semisimplifications of the local systems $\oR^n f_*\R$ for all $n$. Then observe that we have a decomposition
\[
p^*\Delta^* \varpi_1X^{R, \mal} \cong \pi_1^{\dR}(X/Y,p) \rtimes \Delta^* \varpi_1Y^{R, \mal},
\]
where $\pi_1^{\dR}(X/Y,p)$ is Lazda's relative fundamental group from \cite{lazdaRelativeFundamental}.
\end{remark}

\subsection{$\Ql$-homotopy types} \label{qlhtpy}

Take a connected algebraic space $X$ and choose a full rigid tensor subcategory $\cS$ of the category of  semisimple  lisse $\Ql$-sheaves on $X$. Let $\cT$ be the cosimplicial tensor category with the same objects as $\cS$, but with morphisms
\[
 \cT(U,V)= \CC^{\bt}(X, U\ten V^{\vee}),
\]
where $\CC^{\bt}(X,-)$ is the $\ell$-adic Godement resolution of \cite[Definition 2.3]{paper1}. Note that $\H^0\cT \simeq \cS$, and that
\[
 \H^i\cT(U,V)\cong \H^i_{\et}(X, U\ten V^{\vee}),
\]
the $\ell$-adic \'etale cohomology groups.
As in Remark \ref{singDGcat}, we may apply the Thom--Sullivan functor $\Th$ to obtain a dg category $\Th(\cT)$.

Now, any geometric point $\bar{x}$ defines a fibre functor $\bar{x}^* \co \cT^0 \to \FD\Vect_{\Ql}$ sending $U$ to $U_{\bar{x}}$. As in \S \ref{derhamsn}, we may then construct an affine group scheme $R:=\Spec (\bar{x}^*)^{\vee}\ten_{\cS}\bar{x}^*$ over $\Ql$, with  an equivalence $\bar{x}^*\co \cS \to \Rep(R)$ of tensor categories. Equivalently, we have a Zariski-dense continuous group homomorphism $\rho \co \pi_1^{\et}(X,\bar{x}) \to R(\Ql)$. The $R$-equivariant dg algebra $A$ from \S \ref{tensorsubsn} is then just the dg algebra
\[
 \Th \CC^{\bt}(X, \bO(R))
\]
  of equivariant cochains from \cite[Definition \ref{weiln-CCdef}]{weiln}, so $\cT$ is equivalent to $\Rep(R,  \CC^{\bt}(X, \bO(R)) )$.  

\begin{proposition}\label{eqhtpy}
 The dg Hopf algebra $C\simeq (\bar{x}^*)^{\vee}\ten^{\oL}_{\Th(\cT)}\bar{x}^*$ of Corollary \ref{hopfalgconc3} associated to the fibre functor $\bar{x}^*\co \cT \to \FD\Vect$ is a model for the relative Malcev homotopy type $G(X,\bar{x})^{R,\mal}$ of $(X,\bar{x})$ under the equivalences of \cite[Theorem \ref{htpy-bigequiv}]{htpy}.
\end{proposition}
\begin{proof}
 The proof of Proposition \ref{propforms} carries over, replacing \cite[Proposition \ref{htpy-propforms}]{htpy} with \cite[Theorem \ref{weiln-qleqhtpy}]{weiln}.
\end{proof}

Now, we may regard $\per_{\dg}(\cT)$ as the dg subcategory of generated by $\cS$ in the dg category of $\sC^{\bt}_X(\Ql)$-modules in complexes of $\Ql$-sheaves, where $\sC^{\bt}_X$ is the sheaf version of the Godement resolution. Since $\sC^{\bt}_X(\Ql)$ is a flabby resolution of $\Ql$, this means that $\cD(\cT)$ is the derived category of $\Ql$-hypersheaves generated by $\cS$ under extensions, shifts and direct sums.

If $X$ is defined over a separably closed field $\bar{F}$, then $(X\by_{\bar{F}}X)_{\et} \cong X_{\et}\by X_{\et}$, which means that a universal bialgebra $D$ for $(X,\cS)$ corresponds to a bialgebra $\bD$ in the category of  $\Ql$-hypersheaves on $X \by_{\bar{F}} X$.

\begin{remark}\label{symmetry2rk}
Although we are working with \'etale homotopy types rather than Betti homotopy types, the argument of Remark \ref{symmetry1rk} carries over to say that symmetries of $X$ transfer to the universal bialgebra. In particular, this applies to Galois actions.

Explicitly, take an an algebraic space $X_0$ over a field $F$ with separable closure $\bar{F}$, and set $X= X_0\ten_F\bar{F}$. Assume that $\cS$ is generated by pullbacks of lisse sheaves on $X_0$ --- this is equivalent to saying that $\cS$ is $\Gal(F)$-equivariant with finite orbits. Then \cite[Theorem \ref{weiln-lfibrations}]{weiln} ensures that the relative Malcev homotopy type $G(X)^{R, \mal}$ carries a continuous Galois action, so we may regard the universal bialgebra $\bD$  as a Galois-equivariant  $\Ql$-hypersheaf on $X \by_{\bar{F}} X$, or equivalently as a $\Ql$-hypersheaf on $X_0 \by_{F} X_0$. For any basepoint $x \in X_0(F')$, this gives an action of $\Gal(F')$ on the dg bialgebra $(\bar{x}, \bar{x})^*\bD$, but the $\Gal(F)$-action on the universal bialgebra $\bD$ does not require $X(F)$ to be non-empty.
\end{remark}

\begin{remark}\label{generalbaseQl}
As in \cite[Remark \ref{HHtannaka1-generalbase}]{HHtannaka1} and Remark \ref{generalbaseBetti}, we can consider multiple basepoints instead, and taking  the set of all geometric points gives a dg Hopf algebroid $C$ with $C(\bar{x}, \bar{y}) \simeq (\bar{x}^*)^{\vee}\ten^{\oL}_{\Th(\cT)}\bar{y}^*$ as a model for the relative Malcev homotopy type $G(X_{\et})^{R,\mal}$.  
\end{remark}

\subsection{Motivic homotopy types}\label{motsn}

There is nothing special about the de Rham, Betti and $\ell$-adic cohomology theories considered so far in this section. Each construction of pro-algebraic homotopy types has only relied on a suitable sheaf of dg algebras, and a category of projective  modules over it. There are thus analogues for any mixed Weil cohomology theory in the sense of \cite{cisinskidegliseMixedWeil}, or if we are willing to replace Hopf algebras with coalgebras, for any stable cohomology theory. 

\subsubsection{Nilpotent homotopy types}\label{nilpmot}

We now look at the simplest relative Malcev homotopy types, when $R=1$, as in Examples \ref{nilpmalcev}. A mixed Weil cohomology theory $E$ has an associated sheaf $E_X$ of commutative dg algebras on each scheme $X$ over our base field $F$, and we write $E(X):= \Gamma(X, E_X)$. Set $\cS= \FD\Vect_k$ and   $\cT= E(X)\ten \cS$, so $\cT$ has the same objects as $\cS$, but $\cT(U,V)= E(X)\ten \cS(U,V)$.
We may then embed $\cD_{\dg}(\cT^{\op}\ten \cT)= \cD_{\dg}(E(X)^{\op}\ten E(X))$  into the category of $E_{X^2}$-modules by setting 
\[
 \iota_{X^2}(U,V)=U\ten_k E_{X^2}\ten_k V^{\vee}, 
\]
with the left and right actions of $E(X)$ on $E_{X^2}$ coming from the projections $X^2 \to X$.

As in \S \ref{univhalg}, we may now construct a universal Hopf algebra $\bD$ on $X^2=X \by X$, and regard it as the ring of functions on the  space of nilpotent algebraic paths. For an explicit model, we follow Examples \ref{HHcoalgex} and \ref{irredHHrmk2}, setting   
\[
 \bD:= \iota_{X^2} \CCC(\cT/\cS,i^{\op}\ten i )= \CCC(E(X), E_{X^2}),
\]
for $i \co \cS \to \cT$. This is the Hochschild homology complex of the DGA $E(X)$ with coefficients in the $E(X)$-bimodule $E_{X^2}$ in sheaves on $X^2$. 

As before, we have
\[
\bD=  \CCC(E(X), (\pr_1^{-1} E_X) \ten_k (\pr_2^{-1}E_X))\ten_{[(\pr_1^{-1} E_X) \ten_k (\pr_2^{-1}E_X)]} E_{X^2},
\]
which is  defined in terms of the  $(\pr_1^{-1} E_X) \ten_k (\pr_2^{-1}E_X)$-modules 
\[
\uline{\CCC}_n(\cT/\cS,(\pr_1^{-1}\iota_X i^{\op})\ten (\pr_2^{-1}\iota_Xi))\cong\pr_1^{-1} E_X\ten_kE(X)^{\ten n}\ten_k\pr_2^{-1}E_X.
\]
on $X^2$. Since $E$ is a mixed Weil cohomology theory,  this is quasi-isomorphic to
\[
 \pr_1^{-1} E_X\ten_kE(X^n)\ten_k\pr_2^{-1}E_X.
\]

Writing $h(X/Y)$ for the cohomological motive  $M_k(X/Y)^{\op} \in \cM_{\bA^1}(Y)^{\op}$ of $X$ over a base scheme $Y$ and $h(X):= h(X/F)$, we see that the sheaf $\bD$ comes from applying $E$ to the simplicial  motive
\[
n \mapsto  \pr_1^{-1}h(X/X) \ten_k h(X^n)\ten_k\pr_2^{-1}h(X/X)
\]
on $X^2$.

A choice of basepoint $a \co \Spec F \to X$ gives a fibre functor $a^* \co E(X) \to k$, and hence $\cT \to \FD\Vect_k$. The associated dg Hopf algebra is
\[
 C:= (a,a)^*\bD\cong E(X^{\bt}),
\]
with the outer boundary  maps  coming from pulling back by $(a, \id)^*, (\id,a)^*\co X^n \to X^{n+1}$.

Thus $C$ comes from applying the cohomology theory $E$ to the simplicial  motive
\[
\xymatrix@1{ h(F)  \ar@{.>}[r] & \ar@<1ex>[l]^{a^*}  \ar@<-1ex>[l]_{a^*}   h(X) \ar@<1ex>@{.>}[r] \ar@<-1ex>@{.>}[r]& \ar@<2ex>[l]^-{(a, \id)^*} \ar[l]|-{\Delta^*}  \ar@<-2ex>[l]_-{(\id,a)^*}   h(X\by X) \ldots },
\]
which is just $h(\cP_a(X))$, for Wojtkowiak's cosimplicial loop space $\cP_a(X)$ from \cite{wojtkowiakCosimplicial}. The motivic fundamental group of \cite{esnaultlevine} is then essentially just 
\[
 \Spec \H^0 h(\cP_a(X)),   
\]
 so becomes a special case of our Hochschild homology construction for Tannakian duals.

In fact, we can say much more. Following \cite[\S 6]{esnaultlevine}, we define a cosimplicial scheme
$X^{[0,1]}$
by 
$(X^{[0,1]})^n= X^{\Delta^1_n} \cong X^{n+2}$.
The vertices of $\Delta^1$ give a cosimplicial map from  $X^{[0,1]}$ to the constant cosimplicial scheme $X^2$, with $\cP_a(X)$ the fibre over $(a,a)$. Now, observe that the ring of functions $\bD$ on the  space of nilpotent algebraic paths is just given by applying our chosen cohomology theory to the simplicial cohomological motive $\sD:=h(X^I/X^2)$ over $X^2$.  

\subsubsection{Relative Malcev homotopy types}\label{relmalmot}

Rather than just looking at nilpotent homotopy types, we could consider more general motivic homotopy types by choosing a set $S$ of  rigid cohomological motives over $X$, the nilpotent case being $S= \{h(X/X)\}$. Taking $\cT$ to be the full dg category of $E_X$-modules on objects $E_X(M)$ for $M \in S$, 
we find that the universal coalgebra $\bD$ (thought of as the sheaf of functions on the  space of  algebraic paths generated by $S$) is the normalised total complex of the simplicial diagram given in level $n$ by 
\begin{align*}
 & \uline{\CCC}_n(\cT, (\pr_1^{-1}\iota_Xh_{\cT}^{\op})\ten (\pr_2^{-1}\iota_Xh_{\cT}) )\cong\\
 &\bigoplus_{M_0,\ldots, M_n \in S} \pr_1^{-1} E_X(M_0^{\vee}) \ten_kE(M_0\ten_XM_1^{\vee})\ten_k \ldots \ten_k E(M_{n-1}\ten_XM_n^{\vee})\ten_k\pr_2^{-1}E_X(M_n).       
\end{align*}
Here $M^{\vee}$ denotes the dual motive to $M$ over $X$, which is just $M(-d)[-2d]$ when $M=h(Y/X)$ is the motive  of a smooth and proper morphism $Y \to X$ of relative dimension $d$. We write $\ten_X$ for the derived  tensor product of motives over $X$ (i.e. with respect to $k_X:= h(X/X)$), and we set $E(N):= \Gamma(X, E_X(N))$.

Beware that  the duals and tensor products in this expression are only defined up to homotopy, so we have only described $\bD$ as a coalgebra in the derived category of $E$-modules   over $X \by X$, with respect to the tensor product 
\[
 (F,G) \mapsto \pr_{13*}((\pr_{12}^*F)\ten_{E_{X^3}}(\pr_{23}^*G)).
\]

Now, $\bD$ arises by applying $E$ to the  simplicial cohomological motive $\sD$ over $X^2$ given by
\[
 n \mapsto \bigoplus_{M_0, \ldots, M_n \in S} M_0^{\vee} \ten_k (M_0\ten_XM_1^{\vee}) \ten_k  \ldots \ten_k (M_{n-1}\ten_XM_n^{\vee}) \ten_k M_n,       
\]
 where the  $h(X^2)$-module structure  comes from the $h(X)$-module structures of  $M_0^{\vee}$ and $M_n$. 
When the set $S$ is closed under the  tensor product $\ten_X$, the universal coalgebra $\sD$ becomes a  $\ten_X$-bialgebra over $X^2$. 

To understand the relation between $\sD$ and the universal coalgebras of \S \ref{univcoalg}, observe that the six functors formalism of \cite{ayoub6OpsII} 
makes $\cM_{\bA^1}(X)$ a category enriched in $\cM_{\bA^1}(F)$ and linear over it. The $\ten_X$-coalgebra $\sD$ on $X^2$ is then  a resolution of the enriched $\Hom$ functor on objects in $S$ given by $(N,M) \mapsto \oR f_*(M\ten_XN^{\vee})$,  
for $f \co X \to \Spec F$. This construction is thus the direct generalisation of \S \ref{univcoalg} to enriched categories.

At a basepoint $a \in X$, the $E$-Malcev homotopy type of $(X,a)$ relative to $S$  is the dg coalgebra $C:= (a,a)^*\bD$, which just comes from applying $E$ to the simplicial cohomological $F$-motive $(a,a)^*\sD$ given by 
\[
 n \mapsto \bigoplus_{M_0, \ldots, M_n \in S} (M_0^{\vee})_a \ten_k (M_0\ten_XM_1^{\vee}) \ten_k  \ldots \ten_k (M_{n-1}\ten_XM_n^{\vee}) \ten_k (M_n)_a.
\]
In other words, we should think of $\sD$ as the motive of $\cM_{\bA^1}(F)$-valued functions on the  space of  algebraic paths generated by $S$. 
At any basepoint $a$, the motive $(a,a)^*\sD$ is then the geometric motivic homotopy type of $(X,a)$ relative to $S$. Note that the arithmetic homotopy type would replace the motive $(M_{i-1}\ten_XM_i^{\vee}) $ with its motivic cohomology complex.

As in Example \ref{HHcoalgex}, the motive $L:= \bigoplus_{M \in S} M^{\vee}\ten_kM$ is a $\ten_X$-coalgebra over $X^2$. Then $\sD_n = \underbrace{L\ten_XL \ten_X \ldots \ten_XL}_{n+1}$, so $\sD$ is just the \v Cech nerve of the comonoid $L$. Setting $S= \{k_X\}$ (the nilpotent case), we get $L= k_{X^2}=h(X^2/X^2)$ and recover the description $\sD= h(X^I/X^2)$ of \S \ref{nilpmot}.

For rigid motives $P,Q \in \cM_{\bA^1}(X)^{\op}$, we have
\begin{align*}
& \oR\HHom_{ \cM_{\bA^1}^{\op}(X\by X)}( \pr_1^*P^{\vee}\ten_{X \by X} \pr_2^*Q , \sD) \simeq \oR\HHom_{ \cM_{\bA^1}^{\op}(X)}(P,Q) \\
&\simeq\oR\HHom_{ \cM_{\bA^1}^{\op}(X\by X)}( \pr_1^*P^{\vee}\ten_{X \by X} \pr_2^*Q , h(X/X^2)),
\end{align*}
where the morphism $X \to X^2$ is the diagonal map. 
Thus the universal coalgebra $\sD$ is just the universal motive under $h(X/X^2)$ generated by motives in $S$. Since duals and tensor products are here only defined up to homotopy, we should perhaps think of  $h(X/X^2)$ (or at least its induced dg functor on rigid motives)  as the fundamental object. 

When $S$ is the set of all rigid motives and we have a basepoint $a \in X$, $a^*\sD \in \cD_{\bA^1}(F)$ is just Ayoub's motivic Hopf algebra 
from  \cite[\S 2.4]{ayoubGaloisMotivic2}.

\bibliographystyle{alphanum}
\bibliography{references.bib}
\end{document}

%% file: newcommands.tex
\newcommand\la{\leftarrow}

\newcommand\id{\mathrm{id}}

\newcommand\ten{\otimes}

\newcommand\eps{\epsilon}

\newcommand\CC{\mathrm{C}}

\newcommand\Th{\mathrm{Th}\,}
\newcommand\CCC{\mathrm{CC}}

\renewcommand\H{\mathrm{H}}
\newcommand\z{\mathrm{Z}}

\newcommand\HH{\mathrm{HH}}

\newcommand\Z{\mathbb{Z}}
\newcommand\Q{\mathbb{Q}}
\newcommand\Ql{\mathbb{Q}_{\ell}}

\newcommand\R{\mathbb{R}}
\newcommand\Cx{\mathbb{C}}

\newcommand\bA{\mathbb{A}}

\newcommand\bD{\mathbb{D}}

\newcommand\bG{\mathbb{G}}

\newcommand\bO{\mathbb{O}}

\newcommand\C{\mathcal{C}}

\newcommand\cA{\mathcal{A}}
\newcommand\cB{\mathcal{B}}

\newcommand\cD{\mathcal{D}}

\newcommand\cF{\mathcal{F}}
\newcommand\cG{\mathcal{G}}

\newcommand\cM{\mathcal{M}}
\newcommand\cN{\mathcal{N}}

\newcommand\cP{\mathcal{P}}

\newcommand\cR{\mathcal{R}}
\newcommand\cS{\mathcal{S}}
\newcommand\cT{\mathcal{T}}

\newcommand\cZ{\mathcal{Z}}

\newcommand\sA{\mathscr{A}}

\newcommand\sC{\mathscr{C}}
\newcommand\sD{\mathscr{D}}

\newcommand\sF{\mathscr{F}}

\newcommand\sH{\mathscr{H}}

\newcommand\sU{\mathscr{U}}
\newcommand\sV{\mathscr{V}}

\newcommand\Alg{\mathrm{Alg}}

\newcommand\Hopf{\mathrm{Hopf}}
\newcommand\Comm{\mathrm{Comm}}

\newcommand\Hom{\mathrm{Hom}}

\newcommand\HHom{\underline{\mathrm{Hom}}}

\newcommand\End{\mathrm{End}}

\newcommand\dg{\mathrm{dg}}
\newcommand\per{\mathrm{per}}

\newcommand\Gal{\mathrm{Gal}}

\newcommand\Ob{\mathrm{Ob}\,}

\newcommand\Co{\mathrm{Co}}

\newcommand\mal{\mathrm{Mal}}

\newcommand\Spec{\mathrm{Spec}\,}

\newcommand\Cat{\mathrm{Cat}}

\newcommand\Sing{\mathrm{Sing}}
\newcommand\FD{\mathrm{FD}}

\renewcommand\>{\rangle}
\newcommand\Lim{\varprojlim}
\newcommand\LLim{\varinjlim}

\newcommand\into{\hookrightarrow}

\newcommand\abuts{\implies}
\newcommand\xra{\xrightarrow}

\newcommand\pr{\mathrm{pr}}

\newcommand\alg{\mathrm{alg}}

\newcommand\bt{\bullet}
\newcommand\by{\times}

\newcommand\Vect{\mathrm{Vect}}

\newcommand\Rep{\mathrm{Rep}}

\newcommand\Symm{\mathrm{Symm}}

\newcommand\et{\acute{\mathrm{e}}\mathrm{t}}

\newcommand\Nori{\mathrm{Nori}}

\newcommand\mot{\mathrm{mot}}

\newcommand\Tot{\mathrm{Tot}\,}

\renewcommand\ss{\mathrm{ss}}

\newcommand\ind{\mathrm{ind}}

\newcommand\pd{\partial}

\newcommand\MHS{\mathrm{MHS}}
\newcommand\MTS{\mathrm{MTS}}

\newcommand\gr{\mathrm{gr}}

\renewcommand\alg{\mathrm{alg}}
\newcommand\red{\mathrm{red}}

\newcommand\Lie{\mathrm{Lie}}

\newcommand\dR{\mathrm{dR}}

\newcommand\op{\mathrm{opp}}

\newcommand\eff{\mathrm{eff}}

\newcommand\co{\colon\thinspace}

\newcommand\oR{\mathbf{R}}

\newcommand\oL{\mathbf{L}}

\newcommand\uleft\underleftarrow
\newcommand\uline\underline
\newcommand\uright\underrightarrow

%% file: newthms.tex
\newtheorem{theorem}{Theorem}[section]
\newtheorem{proposition}[theorem]{Proposition}
\newtheorem{corollary}[theorem]{Corollary}

\newtheorem{lemma}[theorem]{Lemma}
\newtheorem*{theorem*}{Theorem}
\newtheorem*{proposition*}{Proposition}
\newtheorem*{corollary*}{Corollary}
\newtheorem*{lemma*}{Lemma}
\newtheorem*{conjecture*}{Conjecture}

\theoremstyle{definition}
\newtheorem{definition}[theorem]{Definition}

\newtheorem*{definition*}{Definition}

\theoremstyle{remark}
\newtheorem{example}[theorem]{Example}
\newtheorem{examples}[theorem]{Examples}
\newtheorem{remark}[theorem]{Remark}

\newtheorem*{example*}{Example}
\newtheorem*{examples*}{Examples}
\newtheorem*{remark*}{Remark}
\newtheorem*{remarks*}{Remarks}
\newtheorem*{exercise*}{Exercise}
\newtheorem*{property*}{Property}
\newtheorem*{properties*}{Properties}